\documentclass[reqno]{amsart}

\usepackage{appendix}
\usepackage[utf8]{inputenc}
\usepackage[T1]{fontenc}
\usepackage{lmodern}
\usepackage{hyperref}
\usepackage{amsmath,bm,physics}
\usepackage{amsthm,amssymb}
\usepackage{amsfonts,mathrsfs}
\usepackage{epstopdf}
\usepackage[table]{xcolor}
\usepackage{enumitem}
\usepackage[numbers,sort&compress]{natbib}
\usepackage{graphicx,pstricks,listings,caption,tikz}
\usepackage{soul}
\usepackage{cleveref}
\usepackage {setspace}

\usepackage{geometry}
\allowdisplaybreaks[4]

\graphicspath{ {./my_report_images/} }

\sethlcolor{orange}

\numberwithin{equation}{section}
\newtheorem{corollary}{Corollary}[section]
\newtheorem{definition}{Definition}[section]
\newtheorem{lemma}{Lemma}[section]
\newtheorem{proposition}{Proposition}[section]
\newtheorem{remark}{Remark}[section]

\newtheorem{theorem}{Theorem}[section]

\begin{document}
\title{Periodic Dynamics of Transonic Shocks in Euler--Poisson Flows}

\author{Peng Qu}
\address[Peng Qu]{School of Mathematical Sciences, Fudan University, Shanghai, China and Shanghai Key Laboratory of Contemporary Applied Mathematics, Shanghai, China}
\email{pqu@fudan.edu.cn}

\author{Jiahui Wang}
\address[Jiahui Wang]{School of Mathematical Sciences, Fudan University, Shanghai, China}
\email[corresponding author]{23110180037@m.fudan.edu.cn}

\date{}


\begin{abstract}
	This paper studies the existence and dynamical stability of time--periodic transonic shock solutions to the one--dimensional Euler--Poisson equations under periodic boundary perturbations. The analysis relies on two key mechanisms: the structural dissipation induced by the transonic shock and the stabilization effect of the positive electric field on the shock position. Using a reformulation of the Rankine--Hugoniot conditions and uniform estimates on the flow and shock perturbations, we establish the existence of a time-periodic transonic shock solution. We further prove its dynamical stability under small perturbations of the initial data.
\end{abstract}
\maketitle

\section{Introduction}
We study the following one--dimensional Euler--Poisson equations, which describe charged compressible fluids arising naturally in hydrodynamic descriptions of semiconductor devices and plasmas (\cite{Markowich1991}):
\begin{equation}\label{a1}
		\begin{split}
			\rho_{t}+(\rho u)_{x}&=0,\\
			(\rho u)_{t}+(\rho u^{2}+p)_{x}&=\kappa \rho E,\\
			E_{x} &= \rho - b_{0},
		\end{split}
\end{equation}
where $\rho, u, E$ are the unknowns on $(t,x) \in \mathbb{R}_{+}\times[0,L]$, denoting the density, velocity and the electric field respectively, while $b_{0} > 0$ stands for the positive background charge. We introduce the parameter $\kappa >0$ to measure the strength of the electrostatic force. Similar coupling parameters have also been used in \cite{TadmorWei2008,WeiTadmorBae2012}. We assume that the pressure $p=p(\rho)$ satisfies
\begin{equation*}
	p(0) = p^{\prime}(0) = 0, \qquad p^{\prime}(\rho)>0, ~~p^{\prime\prime} \ge 0,~~\text{for}~\rho > 0, \qquad p(+\infty) = +\infty.
\end{equation*}
A common example is the polytropic gas, for which $p(\rho) = A\rho^{\gamma}$ with $A>0, \gamma > 1$.
\par We consider the initial boundary value problem of \eqref{a1} with the initial data
\begin{align}
	(\rho,u,E)(t,x)\big|_{t=0}&=(\rho_{0},u_{0},E_{0})(x),\label{a3}
\end{align}
and the boundary conditions
\begin{align}
	(\rho,u,E)(t,x)\big|_{x=0}&=(\rho_{b-},u_{b-},E_{b-})(t), \quad \rho(t,x)\big|_{x=L}=\rho_{b+}(t),\label{a5}
\end{align}
satisfying
\begin{align*}
	E_{0}(x) = E_{b-}(0)+ \int_{0}^{x}\left(\rho_{0}(y) - b_{0}\right)dy.
\end{align*}
\par A transonic shock solution for problem \eqref{a1}--\eqref{a5} is defined as follows.
\begin{definition}\label{D1}
	A piecewise $C^1$ solution to~\eqref{a1} is said to be a {\rm{transonic shock solution}} if it is separated by a shock wave positioned at $x=\sigma(t)$, mediating a transition from supersonic upstream flow (left of the shock) to subsonic downstream flow (right of the shock), i.e., it is of the form
	\begin{align*}
		(\rho,u,E)(t,x)=
		\left\{\begin{aligned}
			&(\rho_{-},u_{-},E_{-})(t,x),\quad 0\leq x<\sigma(t),\\
			&(\rho_{+},u_{+},E_{+})(t,x),\quad \sigma(t)<x\leq L,
		\end{aligned}\right.
	\end{align*}
	satisfying
	\begin{equation}
		\left(u - \sqrt{p^{\prime}(\rho)}\right)(t,\sigma(t)-) > 0 > \left(u - \sqrt{p^{\prime}(\rho)}\right)(t,\sigma(t)+),
	\end{equation}
	and the \textsl{Rankine--Hugoniot} conditions at the shock front $x = \sigma(t)$ as
	\begin{align}
		\big[\rho u\big]=\sigma^{\prime}(t)\big[\rho\big], \qquad \big[\rho u^{2} + p\big]=\sigma^{\prime}(t)\big[\rho u\big], \qquad \big[E\big] = 0, \label{hr3}
	\end{align}
	where $[f] := f(t,\sigma(t)+) - f(t,\sigma(t)-)$ denotes the jump across the shock, and the Lax geometric entropy condition
	\begin{align}
		\left(u - \sqrt{p^{\prime}(\rho)}\right)(t,\sigma(t)-) &> \sigma^{\prime}(t) > \left(u - \sqrt{p^{\prime}(\rho)}\right)(t,\sigma(t)+), \label{Lax1}\\
		\left(u + \sqrt{p^{\prime}(\rho)}\right)(t,\sigma(t)+) &> \sigma^{\prime}(t). \label{Lax2}
	\end{align}
\end{definition}
\par We choose a class of steady transonic shock solutions as the background solution. The system for steady solutions becomes \footnote{Applying rescaling $\hat{E} = \sqrt{\kappa}\tilde{E},~~\hat{x} = \sqrt{\kappa}x$, the system of $(\rho, u, \hat{E})$ becomes the standard Euler--Poisson system ($\kappa=1$).} 
\begin{align}\label{a6}
		\begin{split}
			\left(\tilde{\rho}\tilde{u}\right)_{x}&=0,\\
			\left(\tilde{\rho}\tilde{u}^{2}+p(\tilde{\rho})\right)_{x}&= \kappa\tilde{\rho}\tilde{E},\\
			\tilde{E}_{x} &= \tilde{\rho} - b_{0},
		\end{split}
\end{align}
with the boundary data
\begin{align}
	(\tilde{\rho},\tilde{u},\tilde{E})(x)\big|_{x=0}=(\tilde{\rho}_{b-},\tilde{u}_{b-},\tilde{E}_{b-}),\qquad \tilde{\rho}(x)\big|_{x=L}=\tilde{\rho}_{b+}.\label{a7}
\end{align}
The steady Euler--Poisson flows have been studied in many works. Subsonic and supersonic solutions were obtained in \cite{DegondMarkowich1990,DegondMarkowich1993,Markowich1991,PengViolet2006}, while the transonic ones in \cite{AscherEtAl1991}. Gamba \cite{Gamba1992} constructed stationary transonic solutions by a vanishing--viscosity method, and Rosini \cite{Rosini2005} analyzed the phase--plane for all kinds of steady solutions. Later, Luo and Xin \cite{LuoXin2012} studied the existence, uniqueness, and multiplicity of one--dimensional transonic shocks for different boundary data and interval lengths. The stability of steady transonic shocks was also studied. Rosini \cite{Rosini2004} proved local-in-time stability for strong transonic shocks. Luo et al. \cite{LuoRauchXieXin2011} proved structural and dynamical stability and showed the importance of the electric field at the shock for the stability of its position.
\par We define the sonic state $\tilde{\rho} = \tilde{\rho}_{s}$ as the unique positive solution for the equation $p^{\prime}(\tilde{\rho}) = \tilde{\rho}_{b-}^{2}\tilde{u}_{b-}^{2}/\tilde{\rho}^{2}$. In the present paper, we focus on the case when the background charge satisfies $0<b_{0}<\tilde{\rho_{s}}$.  In this case, if the boundary data $(\tilde{\rho}_{b-},\tilde{u}_{b-},\tilde{E}_{b-}), \tilde{\rho}_{b+}$ and the interval length $L$ satisfy some compatibility criteria, then the problem \eqref{a6}--\eqref{a7} admits a unique transonic shock solution with the shock located at $x = \tilde{x}$ (see \cite{LuoXin2012}). We further assume
\begin{equation}\label{restore}
	\tilde{E}(\tilde{x}) > 0,
\end{equation}
which is essential for the stabilization mechanism of the shock position.
\par Moreover, for some $\delta>0$, the supersonic branch can be extended to $[0,\tilde{x}+\delta]$ via local extension of~\eqref{a6} and \eqref{a7}, while the subsonic branch similarly extends to $[\tilde{x}-\delta,L]$. Throughout this paper, the notation $(\tilde{\rho}_{\pm},\tilde{u}_{\pm},\tilde{E}_{\pm})(x)$ represents these extended solutions.
\par We would like to analyze the time--varying shock dynamics. Specifically, we consider the periodic boundary data of the form 
\begin{align}
\rho_{b-}(t)=\tilde{\rho}_{b-}+\overline{\rho}_{b-}(t),~~~u_{b-}(t)=\tilde{u}_{b-}+\overline{u}_{b-}(t),~~~ E_{b-}(t) \equiv \tilde{E}_{b-}, ~~~ \rho_{b+}(t)=\tilde{\rho}_{b+}+\overline{\rho}_{b+}(t),\label{a8}
\end{align}
which is a $C^{1}$ perturbation of \eqref{a7}
\begin{align}
	\|\overline{\rho}_{b-}(t)\|_{C^{1}(\mathbb{R}_{+})}+\|\overline{u}_{b-}(t)\|_{C^{1}(\mathbb{R}_{+})}+\|\overline{\rho}_{b+}(t)\|_{C^{1}(\mathbb{R}_{+})}\leq\epsilon, \label{a13}
\end{align}
for some small constant $\epsilon>0$. We further assume that the boundary states share the same period $\mathcal{T}>0$, i.e.,
\begin{align}
\overline{\rho}_{b-}(t+\mathcal{T})=\overline{\rho}_{b-}(t),~~\overline{u}_{b-}(t+\mathcal{T})=\overline{u}_{b-}(t),
~~\overline{\rho}_{b+}(t+\mathcal{T})=\overline{\rho}_{b+}(t).\label{a10}
\end{align}
\par Time--periodic solutions under periodic boundary data have been studied for several compressible flow models. Matsumura and Nishida \cite{Matsumura} studied periodic solutions for viscous gas equations, and Luo \cite{Luo} considered the piston problem for polytropic gases. For inviscid flows, Yuan \cite{Yuan} constructed time--periodic supersonic solutions to the one--dimensional isentropic Euler equations. Qu \cite{Qu} studied the existence and stability of time-periodic solutions for quasilinear hyperbolic systems with dissipative boundary conditions. For further discussions, one can refer to \cite{Yuw, Yuh, Qup1, ZHANG}. In our problem, the solution contains a moving transonic shock. Thus, the shock position must be determined together with the flow and the electric field.
\par We focus on the existence and stability of time--periodic transonic shock solutions to \eqref{a1}. We further assume that
\begin{align}
0<\kappa \ll 1 ,\label{a11}\\
1 + \mathcal{K} < \dfrac{\tilde{u}_{b-}}{\sqrt{p^{\prime}(\tilde{\rho}_{b-})}} < M_{crit} -\mathcal{K}, \label{AA11}
\end{align}
for some constant $\mathcal{K}>0$ depending on $\kappa$ and the constant $M_{crit}$ determined later in \eqref{left}.
\par The main theorems of this paper are presented as follows. First, we prove the existence of time--periodic transonic shock solutions.
\begin{theorem}[Existence of time-periodic transonic shock solutions]\label{t1}
With supersonic boundary data $\rho_{b-}, u_{b-}$ and $\rho_{b+}$ as \eqref{a8}, under the assumptions \eqref{restore} and \eqref{a11}--\eqref{AA11}, there exist a small constant $\epsilon_1 >0$ and a constant $C_{E} > 0$ such that
for any given $\epsilon\in(0,\epsilon_1)$, and any given boundary perturbations $\overline{\rho}_{b-},\overline{u}_{b-},\overline{\rho}_{b+}\in C^1_{t}$ satisfying \eqref{a13}--\eqref{a10}, there exists an initial datum $(\rho^{(\mathcal{T})}_{0},u^{(\mathcal{T})}_{0},E^{(\mathcal{T})}_{0})(x)$ of the form 
\begin{align}\label{a14}
	(\rho^{(\mathcal{T})}_{0},u^{(\mathcal{T})}_{0},E^{(\mathcal{T})}_{0})(x)=
	\left\{\begin{aligned}
		&(\rho_{0-}^{(\mathcal{T})},u_{0-}^{(\mathcal{T})},E_{0-}^{(\mathcal{T})})(x),\quad 0\leq x<\sigma^{(\mathcal{T})}(0),\\
		&(\rho_{0+}^{(\mathcal{T})},u_{0+}^{(\mathcal{T})},E_{0+}^{(\mathcal{T})})(x),\quad \sigma^{(\mathcal{T})}(0)<x\leq L,
	\end{aligned}\right.
\end{align}
satisfying
\begin{multline}
	|\sigma^{(\mathcal{T})}(0)-\tilde{x}|+\|(\rho_{0-}^{(\mathcal{T})},u_{0-}^{(\mathcal{T})},E_{0-}^{(\mathcal{T})})
	-(\tilde{\rho}_{-},\tilde{u}_{-},\tilde{E}_{-})\|_{C^{1}([0,\sigma^{(\mathcal{T})}(0)])}\\
	+\|(\rho_{0+}^{(\mathcal{T})},u_{0+}^{(\mathcal{T})},E_{0+}^{(\mathcal{T})})-(\tilde{\rho}_{+},\tilde{u}_{+},
	\tilde{E}_{+})\|_{C^{1}([\sigma^{(\mathcal{T})}(0),L])}\leq C_{E}\epsilon,\label{a15}
\end{multline}
such that the initial--boundary value problem~\eqref{a1}--\eqref{a5} admits a transonic shock solution $(\rho^{(\mathcal{T})},u^{(\mathcal{T})},E^{(\mathcal{T})})(t,x)$ on $(t, x)\in \mathbb{R}_{+}\times[0,L]$ containing a shock $x=\sigma^{(\mathcal{T})}(t) \in (0,L)$ with time periodicity
\begin{align}\label{temPeriod}
	(\rho^{(\mathcal{T})},u^{(\mathcal{T})},E^{(\mathcal{T})})(t+\mathcal{T},x) = (\rho^{(\mathcal{T})},u^{(\mathcal{T})},E^{(\mathcal{T})})(t,x), \quad \sigma^{(\mathcal{T})}(t+\mathcal{T})=\sigma^{(\mathcal{T})}(t).
\end{align}
\par Moreover, if the supersonic and subsonic regions are denoted by 
\begin{equation}\label{regions}
	\Omega_{-}^{\mathcal{T}} = \left\{(t,x):t\in \mathbb{R}_{+}, 0\le x \le\sigma^{(\mathcal{T})}(t) \right\}, \quad \Omega_{+}^{\mathcal{T}} = \left\{(t,x):t\in \mathbb{R}_{+}, \sigma^{(\mathcal{T})}(t)\le x \le L \right\},
\end{equation}
the transonic shock solution satisfies
\begin{multline}
\sum_{m=0}^{1}|\partial_{t}^{m}(\sigma^{(\mathcal{T})}(t)-\tilde{x})|+\|(\rho_{-}^{(\mathcal{T})},u_{-}^{(\mathcal{T})},
E_{-}^{(\mathcal{T})})
-(\tilde{\rho}_{-},\tilde{u}_{-},\tilde{E}_{-})\|_{C^{1}(\Omega_{-}^{\mathcal{T}})}\\
+\|(\rho_{+}^{(\mathcal{T})},u_{+}^{(\mathcal{T})},E_{+}^{(\mathcal{T})})-(\tilde{\rho}_{+},\tilde{u}_{+},
\tilde{E}_{+})\|_{C^{1}(\Omega_{+}^{\mathcal{T}})}<C_{E}\epsilon,\label{a16}
\end{multline}
where $(\rho_{\pm}^{(\mathcal{T})},u_{\pm}^{(\mathcal{T})},E_{\pm}^{(\mathcal{T})})$ are the solution in the corresponding regions. 
\end{theorem}
Second, we investigate the dynamical stability of time--periodic transonic shock solutions under small initial perturbations.
\begin{theorem}[Stability of the time-periodic transonic shock solution]\label{t2}
With supersonic boundary data $\rho_{b-}, u_{b-}$ and $\rho_{b+}$ as \eqref{a8}, under the assumptions \eqref{restore} and \eqref{a11}--\eqref{AA11}, there exists a constant $\epsilon_2\in(0,\epsilon_1)$ such that for any given $\epsilon\in(0,\epsilon_2)$, any given boundary perturbations  $\overline{\rho}_{b-},\overline{u}_{b-},\overline{\rho}_{b+}\in C^1$ satisfying \eqref{a13}--\eqref{a10}, and any admissible initial data satisfying the standard $C^{1}$ compatibility conditions
\begin{align}\label{a19}
(\rho_{0},u_{0},E_{0})(x)=
\left\{\begin{aligned}
&(\rho_{0-},u_{0-},E_{0-})(x),\quad 0\leq x<x_{0},\\
&(\rho_{0+},u_{0+},E_{0+})(x),\quad x_{0}<x\leq L,
\end{aligned}\right.
\end{align}
with
\begin{multline}
|x_{0}-\tilde{x}|+\|(\rho_{0-},u_{0-},E_{0-})-(\tilde{\rho}_{-},\tilde{u}_{-},\tilde{E}_{-})\|_{C^{1}([0,x_0])}\\
+\|(\rho_{0+},u_{0+},E_{0+})-(\tilde{\rho}_{+},\tilde{u}_{+},\tilde{E}_{+})\|_{C^{1}([x_0,L])}\leq\epsilon,\label{a20}
\end{multline}
the initial--boundary value problem \eqref{a1}--\eqref{a5} admits a transonic shock solution, which includes the shock wave $x=\sigma(t)$ with $\sigma(0)=x_0,~ \sigma(t)\in(0,L)$. 
\par Moreover, this transonic shock solution exponentially converges to the time--periodic solution driven by the same boundary condition, i.e., there exist constants $C_S>0$, $\zeta\in(0,1)$ and $\mathcal{T}_{0}>0$, such that for all $t>0$,
\begin{align}
&\|(\rho_{-},u_{-},E_{-})(t,\cdot)-(\rho_{-}^{(\mathcal{T})},u_{-}^{(\mathcal{T})},E_{-}^{(\mathcal{T})})(t,\cdot)\|
_{C^{0}([0,\min\{\sigma(t),\sigma^{(\mathcal{T})}(t)\}))}\notag\\
&+\|(\rho_{+},u_{+},E_{+})(t,\cdot)-(\rho_{+}^{(\mathcal{T})},u_{+}^{(\mathcal{T})},E_{+}^{(\mathcal{T})})(t,\cdot)\|
_{C^{0}((\max\{\sigma(t),\sigma^{(\mathcal{T})}(t)\},L])}\notag\\
&+\Big|\sigma(t)-\sigma^{(\mathcal{T})}(t)\Big|+\Big|\partial_{t}\sigma(t)-\partial_{t}\sigma^{(\mathcal{T})}(t)\Big|\leq C_{S}\epsilon\zeta^{\lfloor\frac{t}{\mathcal{T}_{0}}\rfloor},\label{a21}
\end{align}
where $(\rho_{\pm},u_{\pm},E_{\pm})$ and $(\rho_{\pm}^{(\mathcal{T})},u_{\pm}^{(\mathcal{T})},E_{\pm}^{(\mathcal{T})})$ are the solutions in the corresponding regions with initial data \eqref{a19}--\eqref{a20} and \eqref{a14}--\eqref{a15}, respectively.
\end{theorem}
By taking $t \to +\infty$ in \eqref{a21}, the uniqueness result for the time--periodic transonic solution follows directly.
\begin{corollary}
	With boundary data $\rho_{b-}, u_{b-}, \rho_{b+}$ as \eqref{a8}, under the assumptions \eqref{restore} and \eqref{a11}--\eqref{AA11}, there exists a constant $\epsilon_3\in(0,\epsilon_2)$ such that for any given $\epsilon\in(0,\epsilon_3)$, any given boundary perturbations  $\overline{\rho}_{b-},\overline{u}_{b-}, \overline{\rho}_{b+}\in C^1$ satisfying \eqref{a13}--\eqref{a10}, the corresponding time--periodic solution $(\rho^{(\mathcal{T})},u^{(\mathcal{T})},E^{(\mathcal{T})})$ with the corresponding shock $\sigma^{(\mathcal{T})}$ is unique.
\end{corollary}
We provide several remarks regarding \Cref{t1,t2}.
\begin{remark}
	The compatibility conditions for the initial boundary value problems for hyperbolic equations have been discussed in \cite{Majda, Rauch1974}.
\end{remark}
\begin{remark}
	It can be seen in \Cref{s3} that \eqref{a16} can be strengthened to include the estimates for $|\partial_{t}^{2}\sigma|$ and $|\partial_{t}^{2}\sigma - \partial_{t}^{2}\sigma^{(\mathcal{T})}|$ respectively. The estimates of the electric field $E$ can be improved to $C^{2}$. In addition, analogous to \cite{Qu}, one can obtain the corresponding regularity result for the time--periodic solution as well as stabilization around it.
\end{remark}
\begin{remark}
	The results of \Cref{t1,t2} are also true if one imposes small perturbations for the inlet electric field, i.e. $E_{b-}(t)=\tilde{E}_{b-}+\overline{E}_{b-}(t)$ for some periodic $\overline{E}_{b-}$.
\end{remark}
\begin{remark}
	In \cite{LuoRauchXieXin2011}, the dynamical stability of a steady transonic shock solution is considered under small $H^{k+2}$ perturbations. Here, we prove the existence of time--periodic transonic shock solutions under time--periodic perturbations.
\end{remark}
\begin{remark}
	It is interesting to compare these results with previous works \cite{ZHANG,quWang2026}. For the isothermal quasi--one--dimensional Euler flow, it has the same dissipation condition \eqref{AA11} as the Euler--Poisson flow (see \cite{ZHANG}). For non--isentropic flows, a wider range of strong shocks is shown to be dissipative (see \cite{quWang2026}). To the best of the authors' knowledge, this is not due to technical reasons, but rather reveals an essential difference between isentropic flows and non--isentropic ones. 
\end{remark}
\par Besides the Euler--Poisson equations that we concern, transonic shocks have also been studied for quasi--one--dimensional nozzle flows; see \cite{Embid,Liu1,Rauch,XinZ}. Time--periodic transonic shocks under periodic boundary perturbations were recently studied in \cite{ZHANG, quWang2026}. In nozzle flows, the source terms come from the nozzle geometry. In the present problem, the source term is non--local and determined by the solution through the Poisson equation.
\par For multidimensional transonic shock problems, Xin and Yin studied the existence, uniqueness, and stability of transonic shocks in two-- and three--dimensional nozzles \cite{XinYin2005,Xin}. Li et al. further established the existence of multidimensional transonic shocks for the full Euler equations in general nozzles \cite{LiXinYin2010,LiXinYin2013}. More recently, Weng et al. proved the structural stability of three--dimensional axisymmetric transonic shocks under perturbations of the incoming flow, nozzle boundary, and exit pressure \cite{WengXieXin2021}. These works provide important results on multidimensional transonic shocks for the Euler equations.
\par The main difficulty in this analysis is that the shock position evolves as an unknown free boundary. The analysis relies on two stability mechanisms. First, the transonic shock provides an intrinsic structural dissipation for the downstream flow. Second, the positive electric field at the background shock stabilizes the shock position.
\par The proof is based on two key calculations at the shock. We first reformulate the Rankine--Hugoniot conditions in terms of the Riemann variables and calculate the corresponding dissipation coefficient, and then derive the ODE for the shock position and calculate its stability coefficient. The positive electric field at the shock \eqref{restore} gives the required sign of this coefficient and provides control of the shock displacement. Combining these with the weak electrostatic coupling condition \eqref{a11}, we construct the time--periodic transonic shock and establish its stability.
\par This paper is organized as follows: \Cref{s2} gives some useful preliminaries, including the existence of steady transonic shock solutions, the analysis of time--periodic solutions in the supersonic region, and the lemmas on the existence of periodic solutions for a kind of ODE. \Cref{sectionDissi} investigates the electric stabilization mechanism to the shock position and the shock boundary dissipation. Finally, \Cref{t1,t2} are proved in \Cref{s3,s4} respectively.

\section{Preliminaries}\label{s2}
In this section, we present three preliminary results.
\subsection{Basic analysis}\label{su2}
Now, we start with some formal basic analysis of transonic shock solutions for system \eqref{a1}. The sonic speed $c = c(\rho)$ and the Mach number $M$ are defined by
\begin{equation*}
	c(\rho)=\sqrt{p^{\prime}(\rho)}, \quad M=\frac{u}{c}.
\end{equation*}
\par It follows from the third equation in \eqref{a1} that
\begin{equation}\label{electri}
	E(t,x) = E_{b-} + \int_{0}^{x}\left(\rho(t,y) - b_{0}\right)dy.
\end{equation}
Later on, we mainly deal with the first two hyperbolic equations, and consider the electric field as an integral inhomogeneous term. Denote the dynamical unknowns by $U=\left(\rho, u\right)^{\top}$, the fixed background solution by $(\tilde{U}, \tilde{E}) = (\tilde{\rho}, \tilde{u}, \tilde{E})(x)$, the background shock position by $\tilde{x}$, and the perturbations by
\begin{equation*}
	\overline{\rho} = \rho - \tilde{\rho}, \quad \overline{u} = u - \tilde{u}, \quad \overline{E} = E - \tilde{E}.
\end{equation*} 
We choose the Riemann invariants $\mathcal{R} (U) = (\mathcal{R}_{1}, \mathcal{R}_{2})^{\top}(U)$ given by
\begin{equation}\label{newVaria}
	\mathcal{R}_{1} = u - \int_{0}^{\rho}
	\frac{\sqrt{p^{\prime}(s)}}{s}ds,\qquad \mathcal{R}_{2} = u + \int_{0}^{\rho}
	\frac{\sqrt{p^{\prime}(s)}}{s}ds,
\end{equation}
and the perturbation variables $\overline{\mathcal{R}}(x,U) = \left(\overline{\mathcal{R}}_{1}, \overline{\mathcal{R}}_{2}\right)^{\top}$ given by
\begin{equation}\label{upPertur}
	\overline{\mathcal{R}}_{1} = \mathcal{R}_{1}(U) - \tilde{\mathcal{R}}_{1}(\tilde{U}(x)), \qquad \overline{\mathcal{R}}_{2} = \mathcal{R}_{2}(U) - \tilde{\mathcal{R}}_{2}(\tilde{U}(x)).
\end{equation}
\par In order to carry out the variable substitution, we check the Jacobian
\begin{equation}\label{Jacobi}
	\nabla_{U} \overline{\mathcal{R}}\left(x,U\right) := \dfrac{\partial\left(\overline{\mathcal{R}}_{1}, \overline{\mathcal{R}}_{2}\right)}{\partial\left(\rho, u\right)} = 
	\begin{pmatrix}
		-\dfrac{\sqrt{p^{\prime}(\rho)}}{\rho} & 1\\[2ex]
		
		\dfrac{\sqrt{p^{\prime}(\rho)}}{\rho} & 1
	\end{pmatrix}
	,
\end{equation}
and its inverse matrix
\begin{equation}\label{inverVariable}
	\left(\nabla_{U} \overline{\mathcal{R}}\left(x,U\right)\right)^{-1} =\dfrac{\partial\left(\rho, u\right)}{\partial\left(\overline{\mathcal{R}}_{1}, \overline{\mathcal{R}}_{2}\right)} =
	\begin{pmatrix}
		-\dfrac{\rho}{2\sqrt{p^{\prime}(\rho)}} & \dfrac{\rho}{2\sqrt{p^{\prime}(\rho)}}\\[2ex]
		\dfrac{1}{2} & \dfrac{1}{2}
	\end{pmatrix}
	.
\end{equation}
Therefore, one can apply the variable substitution $U = U\left(x,\overline{\mathcal{R}}\right)$, and then get the equivalent system under new variables $\overline{\mathcal{R}}$:
\begin{equation}\label{equivSys}
	\begin{split}
		\partial_{t}\overline{\mathcal{R}}_{1} + \lambda_{1}\partial_{x}\overline{\mathcal{R}}_{1} = \kappa \overline{E} - \left(\lambda_{1}- \tilde{\lambda}_{1}\right)\partial_{x}\tilde{\mathcal{R}}_{1}, \\
		\partial_{t}\overline{\mathcal{R}}_{2} + \lambda_{2}\partial_{x}\overline{\mathcal{R}}_{2} = \kappa \overline{E} - \left(\lambda_{2}- \tilde{\lambda}_{2}\right)\partial_{x}\tilde{\mathcal{R}}_{2}, 
	\end{split}
\end{equation}
where the eigenvalues are
\begin{equation}
	\lambda_{1} = u - c(\rho), \quad \lambda_{2} = u+c(\rho),
\end{equation}
and
\begin{equation}\label{barEle}
	\overline{E}(t,x) = \int_{0}^{x}\left(\rho(t,y) - \tilde{\rho}(y)\right)dy.
\end{equation}
\par Next, we calculate the linear coefficients of the inhomogeneous terms, which dominates the growth strength of the wave along the characteristics. Combining \eqref{Jacobi}--\eqref{equivSys}, one can write the system in the form of wave decomposition
\begin{equation}\label{waveDe}
	\begin{split}
		\partial_{t}\overline{\mathcal{R}}_{1} + \lambda_{1}\partial_{x}\overline{\mathcal{R}}_{1} = \kappa \overline{E} + f_{11}\overline{\mathcal{R}}_{1} + f_{12}\overline{\mathcal{R}}_{2} + f_{1,NL}, \\
		\partial_{t}\overline{\mathcal{R}}_{2} + \lambda_{2}\partial_{x}\overline{\mathcal{R}}_{2} = \kappa \overline{E}+ f_{21}\overline{\mathcal{R}}_{1} + f_{22}\overline{\mathcal{R}}_{2} + f_{2,NL}, 
	\end{split}
\end{equation}
where 
\begin{equation}\label{linearPart}
	\nabla_{\overline{\mathcal{R}}}\bm{f}(x) := \left(f_{ij}(x)\right)_{2\times2} = -\dfrac{\kappa \tilde{E}}{2}
	\begin{pmatrix}
		 \dfrac{1}{\tilde{\lambda}_{1}}\left(1+\dfrac{\tilde{\rho}\tilde{c}^{\prime}(\tilde{\rho})}{\tilde{c}}\right) & \dfrac{1}{\tilde{\lambda}_{1}}\left(1-\dfrac{\tilde{\rho}\tilde{c}^{\prime}(\tilde{\rho})}{\tilde{c}}\right)\\[2ex]
		\dfrac{1}{\tilde{\lambda}_{2}}\left(1-\dfrac{\tilde{\rho}\tilde{c}^{\prime}(\tilde{\rho})}{\tilde{c}}\right) & \dfrac{1}{\tilde{\lambda}_{2}}\left(1+\dfrac{\tilde{\rho}\tilde{c}^{\prime}(\tilde{\rho})}{\tilde{c}}\right)
	\end{pmatrix}
	.
\end{equation}
and $f_{1,NL}, f_{2,NL}$ are nonlinear higher--order terms with respect to $\overline{\rho}, \overline{u}$. At this point, one can see the linear part may cause $\overline{\mathcal{R}}_{1}, \overline{\mathcal{R}}_{2}$ to grow exponentially. In order to curb this growth, we suppose that $\kappa \ll 1$ as in \eqref{a11}. Then one can see that
\begin{equation}
	\Vert \nabla_{\overline{\mathcal{R}}}\bm{f} \Vert < C_{g}\kappa.
\end{equation}
The constant $C_{g}$ depends only on the background solution. Meanwhile, 
\begin{equation}\label{Ex}
	\overline{E}_{x} = \dfrac{\tilde{\rho}}{2\tilde{c}}\left(\overline{\mathcal{R}}_{2} - \overline{\mathcal{R}}_{1}\right) + \text{nonlinear terms.}
\end{equation}
\par In the end, we rewrite boundary conditions under new variables $\overline{\mathcal{R}}$ as
\begin{align}
	\overline{\mathcal{R}}_{1}(t,0)=\overline{\mathcal{R}}_{1,b-}(t),\quad \overline{\mathcal{R}}_{2}(t,0)=\overline{\mathcal{R}}_{2,b-}(t),\label{inletBoundary}\\
	\overline{\mathcal{R}}_{1}(t,L)= \overline{\mathcal{R}}_{2}(t,L) + \overline{\mathcal{R}}_{1,b+}(t), \label{outletBoundary}
\end{align}
with some constant $C_{b}>0$ such that
\begin{align}\label{boundarySmall}
	\|\overline{\mathcal{R}}_{1,b-}\|_{C^{1}(\mathbb{R}_{+})}+\|\overline{\mathcal{R}}_{2,b-}\|_{C^{1}(\mathbb{R}_{+})}+\|\overline{\mathcal{R}}_{1,b+}\|_{C^{1}(\mathbb{R}_{+})}\leq C_{b}\epsilon,
\end{align}
and $\mathcal{T}$--periodicity
\begin{align}
	\overline{\mathcal{R}}_{i,b-}(t+\mathcal{T})=\overline{\mathcal{R}}_{i,b-}(t), \quad i = 1,2; \quad \overline{\mathcal{R}}_{1,b+}(t+\mathcal{T})=\overline{\mathcal{R}}_{1,b+}(t). \label{boundaryPeriodicity}
\end{align}

\subsection{Time--periodic solutions in the supersonic region}
We consider flows in the supersonic region, for which both eigenvalues are positive. The lemma below follows from the argument in \cite{Yuw, Yuan}.
\begin{lemma}\label{existSuper}
	There exists a constant $\epsilon_{0} > 0$ such that for any $\epsilon \in (0,\epsilon_{0})$ and any inlet condition $(\rho_{b-}, u_{b-}, E_{b-})$ of the form \eqref{a8} satisfying \eqref{a13} and \eqref{a10}, the boundary problem \eqref{a1} and \eqref{a8} admits a unique time--periodic solution $(\rho_{-}^{(\mathcal{T})},u_{-}^{(\mathcal{T})},E_{-}^{(\mathcal{T})})(t,x)$ on the extended region $(t,x) \in \mathbb{R}_{+}\times[0,\tilde{x}+\delta]$ for some $\delta>0$ satisfying
	\begin{align}
		&\rho_{-}^{(\mathcal{T})}(t+\mathcal{T},x)=\rho_{-}^{(\mathcal{T})}(t,x), ~~ u_{-}^{(\mathcal{T})}(t+\mathcal{T},x)=u_{-}^{(\mathcal{T})}(t,x),
		~~ E_{-}^{(\mathcal{T})}(t+\mathcal{T},x)=E_{-}^{(\mathcal{T})}(t,x),\label{c1}\\
		&\|\rho_{-}^{(\mathcal{T})}(t,x)-\tilde{\rho}_{-}(x)\|_{C^{1}}+\|u_{-}^{(\mathcal{T})}(t,x)-\tilde{u}_{-}(x)\|_{C^{1}}
		+\|E_{-}^{(\mathcal{T})}(t,x)-\tilde{E}_{-}(x)\|_{C^{1}}<C_{l}\epsilon.\label{c2}
	\end{align}
\end{lemma}
The proof is based on the existence of semi--global classical solutions for hyperbolic conservation laws. We omit the details here. Indeed, one can exchange the roles of $t$ and $x$, and periodically extend the boundary conditions. In this way, the problem is transformed into a Cauchy problem for a hyperbolic conservation law. Then, the lemma can be proved by the classical characteristic method and an iteration argument.

\subsection{Periodic shocks driven by periodic perturbations}\label{su11}
In this subsection, we consider the ODE systems of the following form, which describe the dynamics of the shock position under time--periodic perturbations
\begin{align}
	\frac{d\sigma(t)}{dt}=\Sigma(\sigma(t),\varsigma_{1}(t,\sigma(t)),\varsigma_{2}(t,\sigma(t)),\varsigma_{3}(t,\sigma(t)),\varsigma_{4}(t,\sigma(t))),\label{bb1}
\end{align}
where $\varsigma_{i}(t,x)~(i=1,2,3,4)$ are $C^{1}$ perturbation functions on $\mathbb{R}\times[0,L]$ satisfying
\begin{align}
	\varsigma_{i}(t+\mathcal{T},x)=\varsigma_{i}(t,x),\quad i \in \{1,2,3,4\}.\label{TB1}
\end{align} 
Moreover, we assume $\Sigma(\sigma,\varsigma_{1},\varsigma_{2},\varsigma_{3},\varsigma_{4})$ is a smooth function satisfying
\begin{align}
	\Sigma(\tilde{x},0,0,0,0)=0, \qquad \frac{\partial\Sigma}{\partial\sigma}(\tilde{x},0,0,0,0)<0.\label{bb3}
\end{align}
For simplicity, we denote $\Sigma(\sigma,\varsigma_{1},\varsigma_{2},\varsigma_{3},\varsigma_{4})$ by $\Sigma(\sigma, \bm{\varsigma})$ and
$$\Sigma_{\varsigma_{i,0}}=\Big|\frac{\partial\Sigma}{\partial\varsigma_{i}}(\tilde{x},\bm{0})\Big|,~(i=1,2,3,4),
\quad\Sigma_{\sigma,0}=\Big|\frac{\partial\Sigma}{\partial\sigma}(\tilde{x},\bm{0})\Big|.$$
Under these assumptions, the following two lemmas hold. Their proofs are omitted; see the appendix of \cite{ZHANG}. Here and throughout the paper, we denote the $C^{0}$ norm by $\|\cdot\|$ for simplicity.
\begin{lemma}[Existence and uniqueness of the periodic solution]\label{lemmaODE1}
	There exist constants $\varepsilon_{1}>0$ and $C_\Sigma>0$ such that for any perturbations satisfying $\|\varsigma_{i}\|_{C^{1}}<\varepsilon_{1}, i \in \{1,2,3,4\}$, system~\eqref{bb1} admits a unique periodic solution $\sigma(t)$ with $\sigma(t+\mathcal{T})=\sigma(t)$.
	Furthermore, the following estimates hold
	\begin{align}
		\|\sigma - \tilde{x}\|\leq(1+C_{\Sigma}\varepsilon_{1})\frac{1}{\Sigma_{\sigma,0}}\sum_{i=1}^{4}\Sigma_{\varsigma_{i,0}}\|\varsigma_{i}\|,\label{bb19}
	\end{align}
	\begin{align}
		\|{\sigma}^{\prime}\|\leq
		(2+C_{\Sigma}\varepsilon_{1})\sum_{i=1}^{4}\Sigma_{\varsigma_{i,0}}\|\varsigma_{i}\|,\label{bb20}
	\end{align}
	\begin{align}
		\|{\sigma}^{\prime\prime}\|\leq&(2+C_{\Sigma}\varepsilon_{1})\Sigma_{\sigma,0}\sum_{i=1}^{4}\Sigma_{\varsigma_{i,0}}\|\varsigma_{i}\|
		+(1+C_{\Sigma}\varepsilon_{1})\sum_{i=1}^{4}\Sigma_{\varsigma_{i,0}}
		\bigg\Vert\frac{\partial\varsigma_{i}}{\partial t}\bigg\Vert.\label{bb21}
	\end{align}
\end{lemma}
\begin{lemma}[$C^{1}$ stability estimates with respect to perturbations]\label{lemmaODE2}
	Let $\sigma_{1}(t)$ and $\sigma_{2}(t)$ be two periodic solutions driven by different groups of perturbation functions ${\varsigma_{1}}_{i}(t,\sigma)(i=1,2,3,4)$ and ${\varsigma_{2}}_{i}(t,\sigma)(i=1,2,3,4)$ with the same period, i.e.,
	\begin{align}
		\frac{d\sigma_{j}}{dt}(t)=\Sigma(\sigma_{j}(t), \bm{\varsigma_{j}}(t,\sigma_{j}(t))), \quad j=1,2,\label{bb29}
	\end{align}
	and
	\begin{align*}
		\sigma_{1}(t+\mathcal{T})=\sigma_{1}(t),~~\sigma_{2}(t+\mathcal{T})=\sigma_{2}(t),\quad \forall t\in\mathbb{R}.
	\end{align*}
	Then there exists a constant $\varepsilon_{2}>0$ such that for any perturbations satisfying $\|{\varsigma_{1}}_{i}\|_{C^{1}}\leq\varepsilon_{2},~\|{\varsigma_{2}}_{i}\|_{C^{1}}\leq\varepsilon_{2},$
	the following estimates hold
	\begin{align}
		\|\sigma_{1}-\sigma_{2}\|\leq(1+C_{\Sigma}\mathcal{T}\varepsilon_{2})\exp(\Sigma_{\sigma,0}\mathcal{T})\frac{1}{\Sigma_{\sigma,0}}
		\sum_{i=1}^{4}\Sigma_{\varsigma_{i,0}}\|{\varsigma_{1}}_{i}-{\varsigma_{2}}_{i}\|,
		\label{bb33}
	\end{align}
	\begin{align}
		\|{{\sigma}_{1}}^{\prime}-{{\sigma}_{2}}^{\prime}\|
		\leq&(1+C_{\Sigma}\mathcal{T}\varepsilon_{2})\exp(\Sigma_{\sigma,0}\mathcal{T})\sum_{i=1}^{4}\Sigma_{\varsigma_{i,0}}\|{\varsigma_{1}}_{i}
		-{\varsigma_{2}}_{i}\|
		+(1+C_{\Sigma}\varepsilon_{2})\sum_{i=1}^{4}\Sigma_{\varsigma_{i,0}}\|{\varsigma_{1}}_{i}-{\varsigma_{2}}_{i}\|.
		\label{bb34}
	\end{align}
\end{lemma}

\section{Electric Stabilization Mechanism and Shock Dissipation}\label{sectionDissi}
In this section, we investigate the stabilization mechanism of the positive background electric field on the shock position as well as the shock boundary dissipation.
\subsection{Shock boundary conditions}\label{boundsec}
We start by deriving the shock boundary conditions for the new set of variables. It is known that in the subsonic domain, the 1--wave $\mathcal{R}_{1}$ reaches the shock boundary from the right, while the 2--wave $\mathcal{R}_{2}$ is generated at the shock. This relation follows from the R--H condition
\begin{equation}\label{RH}
	\big[\rho u\big]=\sigma^{\prime}(t)\big[\rho\big], \qquad \big[\rho u^{2} + p\big]=\sigma^{\prime}(t)\big[\rho u\big], \qquad \big[E\big] = 0. 
\end{equation}
Here and below we denote the left/right limits of some function $f$ at the shock $\sigma(t)$ by $f_{l} = f(\sigma(t)-)$, $f_{r} = f(\sigma(t)+)$ and the jump across the shock by $[f] = f_{r} - f_{l}$. For the background solution, we denote $\tilde{f}_{l} = \tilde{f}(\tilde{x}-)$, $\tilde{f}_{r} = \tilde{f}(\tilde{x}+)$.
\par The first two equations in \eqref{RH} yield
\begin{equation}\label{RHcon}
	\sqrt{p(\rho_{r})-p(\rho_{l})}\sqrt{\rho_{r}-\rho_{l}} = \sqrt{\rho_{r}\rho_{l}}\left(u_{l}-u_{r}\right).
\end{equation}
We define
\begin{equation}\label{rh1}
	H\left(\rho_{l}, u_{l}, \rho_{r}, u_{r}\right)
	:= \sqrt{p(\rho_{r})-p(\rho_{l})}\sqrt{\rho_{r}-\rho_{l}} - \sqrt{\rho_{r}\rho_{l}}\left(u_{l}-u_{r}\right).
\end{equation}
Using the change of variables, one can rewrite \eqref{RHcon} as
\begin{equation}\label{rh2}
	G\left(x,\overline{\mathcal{R}}_{1}, \overline{\mathcal{R}}_{2}, \overline{\rho}_{l}, \overline{u}_{l}\right) \equiv H\left(\rho_{l}, u_{l}, \rho_{r}, u_{r}\right) \equiv 0.
\end{equation}
For a fixed background solution $(\tilde{\rho}, \tilde{u})(x)$ and known left states $(\rho_{l}, u_{l})$, there are two unknown variables $(\rho_{r}, u_{r})$ in the equation, equivalently the perturbations $(\overline{\rho}_{r}, \overline{u}_{r})$ or $\left(\overline{\mathcal{R}}_{1}, \overline{\mathcal{R}}_{2}\right)$.
\par For a given steady transonic shock solution $\left(\tilde{\rho}, \tilde{u}\right)$, which satisfies the steady R--H conditions
\begin{equation}\label{rhBack}
	\tilde{\rho}_{r}\tilde{u}_{r} = \tilde{\rho}_{l}\tilde{u}_{l} =: m, \qquad \tilde{\rho}_{r}\tilde{u}_{r}^{2} + p(\tilde{\rho}_{r}) = \tilde{\rho}_{l}\tilde{u}_{l}^{2} + p(\tilde{\rho}_{l}),
\end{equation}
one can obtain by direct calculation that
\begin{align}
	V := \dfrac{\sqrt{p(\tilde{\rho}_{r})-p(\tilde{\rho}_{l})}}{\sqrt{\tilde{\rho}_{r}-\tilde{\rho}_{l}}} = \frac{m}{\sqrt{\tilde{\rho}_{r}\tilde{\rho}_{l}}}, \label{rhFormu2}\\
	p(\tilde{\rho}_{r})-p(\tilde{\rho}_{l}) = -m\left(\tilde{u}_{r}-\tilde{u}_{l}\right) = \frac{\tilde{\rho}_{r}-\tilde{\rho}_{l}}{\tilde{\rho}_{r}\tilde{\rho}_{l}}m^{2}. \label{rhFormu3}
\end{align}
\par From \eqref{rh1} and with the help of \eqref{rhFormu2}--\eqref{rhFormu3}, one can obtain
\begin{equation}\label{devi}
		\dfrac{\partial H}{\partial\left(\rho_{r}, u_{r}\right)}\bigg|_{U=\tilde{U}} 
		= 
		\begin{pmatrix}
			 \dfrac{p^{\prime}(\tilde{\rho}_{r})}{2V} + \dfrac{\tilde{\rho}_{l}V}{2\tilde{\rho}_{r}},& \sqrt{\tilde{\rho}_{r}\tilde{\rho}_{l}}
		\end{pmatrix}
		,
		\quad
		\dfrac{\partial H}{\partial\left(\rho_{l}, u_{l}\right)}\bigg|_{U=\tilde{U}} 
		= -
		\begin{pmatrix}
			\dfrac{p^{\prime}(\tilde{\rho}_{l})}{2V} + \dfrac{\tilde{\rho}_{r}V}{2\tilde{\rho}_{l}},& \sqrt{\tilde{\rho}_{r}\tilde{\rho}_{l}}
		\end{pmatrix}
		.
\end{equation}
Combining \eqref{inverVariable} and \eqref{devi}, one can see
\begin{equation}
	\dfrac{\partial G}{\partial\left(\overline{\mathcal{R}}_{1}, \overline{\mathcal{R}}_{2}\right)}\bigg|_{\overline{\mathcal{R}}=0} = \dfrac{\partial H}{\partial\left(\rho_{r}, u_{r}\right)}\bigg|_{U=\tilde{U}} \cdot \dfrac{\partial\left(\rho_{r}, u_{r}\right)}{\partial\left(\overline{\mathcal{R}}_{1}, \overline{\mathcal{R}}_{2}\right)} \bigg|_{\overline{\mathcal{R}}=0} = \dfrac{\tilde{\rho}_{r}}{4V\tilde{c}_{r}}
	\begin{pmatrix}
		 -\left(\tilde{u}_{r}-\tilde{c}_{r}\right)^{2}, & \left(\tilde{u}_{r}+\tilde{c}_{r}\right)^{2}
	\end{pmatrix}
	.
\end{equation}
Thus, one can write the shock boundary condition as
\begin{equation}\label{bounrdayshock}
	\overline{\mathcal{R}}_{2}(t,\sigma(t)) = \mathscr{A}\left(\sigma(t), \overline{\mathcal{R}}_{1}(t,\sigma(t)), \overline{\rho}_{l}(t,\sigma(t)), \overline{u}_{l}(t,\sigma(t))\right),
\end{equation}
where $\mathscr{A} = \mathscr{A}(x, \overline{\mathcal{R}}_{1}, \overline{\rho}_{l}, \overline{u}_{l})$ is regarded as the implicit function determined by \eqref{rh2}. 
\par Next, we focus on the offset of the shock position $\sigma(t)$. Again, from the R--H conditions, one can derive that the shock position satisfies an ODE
\begin{equation}
	\sigma^{\prime}(t) = u_{l}-\sqrt{\dfrac{\rho_{r}}{\rho_{l}}}\dfrac{\sqrt{p(\rho_{r})-p(\rho_{l})}}{\sqrt{\rho_{r}-\rho_{l}}} .
\end{equation}
In particular, we define
\begin{equation}\label{speedEx}
	\mathcal{F}\left(x, \overline{\mathcal{R}}_{1}, \overline{\mathcal{R}}_{2}, \overline{\rho}_{l}, \overline{u}_{l}\right) \equiv F\left(U_{r},U_{l}\right) := u_{l}-\sqrt{\dfrac{\rho_{r}}{\rho_{l}}}\dfrac{\sqrt{p(\rho_{r})-p(\rho_{l})}}{\sqrt{\rho_{r}-\rho_{l}}}.
\end{equation}
And then, one has the ODE of the form \eqref{bb1}
\begin{equation}\label{shockODE}
	\sigma^{\prime}(t) = \mathcal{F}\left(\sigma(t), \overline{\mathcal{R}}_{1}(t,\sigma(t)), \overline{\mathcal{R}}_{2}(t,\sigma(t)), \overline{\rho}_{l}(t,\sigma(t)), \overline{u}_{l}(t,\sigma(t))\right).
\end{equation}

\subsection{Shock stabilization and wave dissipation}
In this subsection, we verify the shock structural dissipation (\cref{dissipationCon}). Here and throughout the rest of this subsection, we discuss a fixed background solution $(\tilde{\rho}, \tilde{u}, \tilde{E}; \tilde{x})$ and the value of functions at the shock jump. For the sake of simplicity, we omit all the value notations $|_{\overline{\mathcal{R}}=0}$ or $\left(\tilde{x},0,0,0,0\right)$.
\par We will estimate the magnitudes of 2--waves generated at the shock by Hadamard's formula from \eqref{bounrdayshock} as
\begin{equation}\label{esti1}
	\overline{\mathcal{R}}_{2} \sim \left|\frac{\partial\mathscr{A}}{\partial x}\right|\Vert\sigma(t) - \tilde{x}\Vert + \left|\frac{\partial\mathscr{A}}{\partial\overline{\mathcal{R}}_{1}}\right|\Vert\overline{\mathcal{R}}_{1}\Vert + \text{supersonic perturbation terms}.
\end{equation}
The two main terms are the resonance term and the reflection term. First of all, one can derive that
\begin{equation}\label{reflecPart}
	\dfrac{\partial \mathscr{A}}{\partial\overline{\mathcal{R}}_{1}} = -\left(\dfrac{\partial G}{\partial\overline{\mathcal{R}}_{1}}\right)\left(\dfrac{\partial G}{\partial\overline{\mathcal{R}}_{2}}\right)^{-1} =  \left(\dfrac{\tilde{u}_{r}-\tilde{c}_{r}}{\tilde{u}_{r}+\tilde{c}_{r}}\right)^{2}.
\end{equation}
To calculate the derivatives with respect to $x$, one may apply the ODEs for the background solution
\begin{equation}\label{backODE}
		\tilde{\rho}^{\prime} = \kappa \dfrac{\tilde{\rho}\tilde{E}}{\tilde{c}^{2}-\tilde{u}^{2}}, \qquad
		\tilde{u}^{\prime} = -\kappa \dfrac{\tilde{u}\tilde{E}}{\tilde{c}^{2}-\tilde{u}^{2}}.
\end{equation}
Combining \eqref{devi} and \eqref{backODE}, one can obtain
\begin{equation}
	\dfrac{\partial G}{\partial x} = \dfrac{\partial H}{\partial\left(\rho_{r}, u_{r}\right)}\cdot \dfrac{\partial\left(\rho_{r}, u_{r}\right)}{\partial x} + \dfrac{\partial H}{\partial\left(\rho_{l}, u_{l}\right)}\cdot \dfrac{\partial\left(\rho_{l}, u_{l}\right)}{\partial x} = \dfrac{\kappa \tilde{E}}{2V}\left(\tilde{\rho}_{r} - \tilde{\rho}_{l}\right).
\end{equation}
Thus, one can get
\begin{equation}\label{pA3px}
	\dfrac{\partial \mathscr{A}}{\partial x} = -\left(\dfrac{\partial G}{\partial x}\right)\left(\dfrac{\partial G}{\partial\overline{\mathcal{R}}_{2}}\right)^{-1} = -2\kappa \tilde{E} \tilde{c}_{r} \dfrac{\tilde{\rho}_{r} - \tilde{\rho}_{l}}{\tilde{\rho}_{r}\left(\tilde{u}_{r}+\tilde{c}_{r}\right)^{2}}.
\end{equation}
\par Furthermore, we will estimate the amplitude of the shock position by \eqref{bb19} as
\begin{equation}\label{estiODE}
	\Vert\sigma(t) - \tilde{x}\Vert \sim \dfrac{1}{\left|\frac{\partial \mathcal{F}}{\partial x}\right|}\sum_{i=1}^{2}\left|\frac{\partial \mathcal{F}}{\partial \overline{\mathcal{R}}_{i}}\right|\Vert\overline{\mathcal{R}}_{i}\Vert+ \text{supersonic perturbation terms}.
\end{equation}
Combining \eqref{speedEx} and \eqref{backODE}, one can derive
\begin{equation}\label{pFpx}
	\frac{\partial \mathcal{F}}{\partial x} = -\dfrac{\kappa\tilde{E}}{2\tilde{u}_{r}}.
\end{equation}
This, combining \eqref{restore} the fact $\mathcal{F}(\tilde{x},\mathbf{0}) = 0$, enables one to apply \Cref{lemmaODE1,lemmaODE2} to the ODE \eqref{shockODE}. More importantly, this reveals the role of the electric field in the dynamical stability of the shock.
\par On the other hand, direct calculation from \eqref{speedEx} yields
\begin{equation}
	\frac{\partial \mathcal{F}}{\partial \rho_{r}} = \dfrac{\tilde{u}_{r}^{2} - \tilde{c}_{r}^{2}}{2\tilde{u}_{r}\left(\tilde{\rho}_{r} - \tilde{\rho}_{l}\right)}.
\end{equation}
Combining \eqref{inverVariable}, one can get
\begin{equation}\label{pFpU1}
	\frac{\partial \mathcal{F}}{\partial \overline{\mathcal{R}}_{1}} = \frac{\partial \mathcal{F}}{\partial \rho_{r}}\cdot\frac{\partial \rho_{r}}{\partial \overline{\mathcal{R}}_{1}}= - \dfrac{\left(\tilde{u}_{r}^{2} - \tilde{c}_{r}^{2}\right)\tilde{\rho}_{r}}{4\tilde{u}_{r}\tilde{c}_{r}\left(\tilde{\rho}_{r} - \tilde{\rho}_{l}\right)}, \quad \frac{\partial \mathcal{F}}{\partial \overline{\mathcal{R}}_{2}} =  \dfrac{\left(\tilde{u}_{r}^{2} - \tilde{c}_{r}^{2}\right)\tilde{\rho}_{r}}{4\tilde{u}_{r}\tilde{c}_{r}\left(\tilde{\rho}_{r} - \tilde{\rho}_{l}\right)}.
\end{equation}
\par In view of \eqref{esti1} and \eqref{estiODE}, our goal is that the coefficient carries dissipation, i.e.,
\begin{equation}\label{dissipationCon}
	\tilde{\theta}_{d} := \left|\frac{\partial\mathscr{A}}{\partial x}\right|\cdot\left|\frac{\partial \mathcal{F}}{\partial x}\right|^{-1}\cdot\left(\left|\frac{\partial \mathcal{F}}{\partial \overline{\mathcal{R}}_{1}}\right|+ \left|\frac{\partial \mathcal{F}}{\partial \overline{\mathcal{R}}_{2}}\right|\right) +  \left|\frac{\partial\mathscr{A}}{\partial\overline{\mathcal{R}}_{1}}\right| < 1.
\end{equation}
Substituting \eqref{reflecPart}, \eqref{pA3px}, \eqref{pFpx} and \eqref{pFpU1} into \eqref{dissipationCon} and taking advantage of \eqref{rhFormu2}--\eqref{rhFormu3}, one can obtain that \eqref{dissipationCon} is equivalent to
	\begin{equation}\label{right}
		\sqrt{2} - 1 < \tilde{M}_{r} := \dfrac{\tilde{u}_{r}}{\tilde{c}_{r}} < 1.
	\end{equation}
The dissipative condition \eqref{AA11} can be obtained as follows. Firstly, for fixed $J = \tilde{\rho}_{b-}\tilde{u}_{b-}$, we define
\begin{equation}
	F_{J}(\rho) := p(\rho) +\dfrac{J^{2}}{\rho}.
\end{equation}
The R--H condition yields that $F_{J}(\tilde{\rho}_{l}) = F_{J}(\tilde{\rho}_{r})$ and determines a function $\varrho(\rho)$ such that
\begin{equation*}
	F_{J}(\varrho(\tilde{\rho}_{l})) = F_{J}(\tilde{\rho}_{l}), \quad \text{with}\quad \varrho^{\prime}(\tilde{\rho}_{l}) = \dfrac{F_{J}^{\prime}(\tilde{\rho}_{l})}{F_{J}^{\prime}(\tilde{\rho}_{r})} < 0.
\end{equation*}
On the other hand, one can see the Mach number satisfies
\begin{equation}
	M = Q_{J}(\rho) := \dfrac{J}{\rho \sqrt{p^{\prime}(\rho)}}, \quad \text{with}\quad Q_{J}^{\prime}(\rho) = - Q_{J}(\rho)\left(\dfrac{1}{\rho} +\dfrac{p^{\prime\prime}(\rho)}{2p^{\prime}(\rho)}\right)<0.
\end{equation}
By the chain rule, one can see the monotonicity
\begin{equation}
	\dfrac{d \tilde{M}_{r}}{d \tilde{M}_{l}} = \dfrac{Q_{J}^{\prime}(\tilde{\rho}_{r})F_{J}^{\prime}(\tilde{\rho}_{l})}{Q_{J}^{\prime}(\tilde{\rho}_{l})F_{J}^{\prime}(\tilde{\rho}_{r})} <0.
\end{equation}
Therefore, by the implicit function theorem, one can write the relation
\begin{equation}
	\tilde{M}_{l} = \mathcal{M}(\tilde{M}_{r}).
\end{equation}
And the condition \eqref{right} is equivalent to 
\begin{equation}\label{left}
	1 < \tilde{M}_{l} < M_{crit} := \lim\limits_{M \to (\sqrt{2} - 1)+ }\mathcal{M}(M).
\end{equation}
Combining with $\kappa \ll 1$, one can see that the condition \eqref{left} can be guaranteed by \eqref{AA11}.
\par In fact, by continuity and boundedness of the background solution, under the assumption \eqref{AA11} we have a stronger conclusion that there exists a constant $\theta_{d}$ independent of $\kappa$ such that
\begin{equation*}
	\tilde{\theta}_{d} < \theta_{d} <1.
\end{equation*}

\subsection{Reformulation of the system}
In order to make the boundary conditions have formal dissipation, we apply the following scaling transformation
\begin{equation}\label{scaling}
	\breve{\bm{\mathcal{R}}} = \left(\breve{\mathcal{R}}_{1}, \breve{\mathcal{R}}_{2}\right)^{\top} := \left(\alpha_{1}\overline{\mathcal{R}}_{1}, \alpha_{2}\overline{\mathcal{R}}_{2}\right)^{\top},
\end{equation}
where we set $\alpha_{1}=1$, and that $\alpha_{2}$ is a constant satisfying
\begin{equation}\label{beta}
	1 < \alpha_{2}<\theta_{d}^{-1}.
\end{equation}
\par Now, we conclude that in the subsonic domain, the problem takes the following form
\begin{align}
	\partial_{t}\breve{\mathcal{R}}_{1}+\lambda_{1}\partial_{x}\breve{\mathcal{R}}_{1} &= \kappa \overline{E} + g_{1}(x, \breve{\bm{\mathcal{R}}}) := \kappa \overline{E} + f_{11}\overline{\mathcal{R}}_{1} + f_{12}\overline{\mathcal{R}}_{2} + f_{1,NL}, \label{sys57}\\
	\dfrac{1}{\alpha_{2}}\partial_{t}\breve{\mathcal{R}}_{2}+\dfrac{1}{\alpha_{2}}\lambda_{2}\partial_{x}\breve{\mathcal{R}}_{2} &= \kappa \overline{E} + g_{2}(x, \breve{\bm{\mathcal{R}}}) := \kappa \overline{E} + f_{21}\overline{\mathcal{R}}_{1} + f_{22}\overline{\mathcal{R}}_{2} + f_{2,NL}, \label{system3}
\end{align}
where $\overline{E}$ is an integral functional of $\breve{\bm{\mathcal{R}}}$. The shock boundary and outlet boundary conditions are presented as
\begin{align}
	\breve{\mathcal{R}}_{2}\left(t, \sigma(t)\right) &=  \alpha_{2}\mathscr{A}\left(\sigma(t), \breve{\mathcal{R}}_{1}(t,\sigma(t)), \overline{\rho}_{l}(t,\sigma(t)), \overline{u}_{l}(t,\sigma(t)) \right), \label{shockBoundary2}\\
	\breve{\mathcal{R}}_{1}\left(t, L\right) &= \frac{1}{\alpha_{2}}\breve{\mathcal{R}}_{2}\left(t, L\right) + \overline{\mathcal{R}}_{1,b+}(t). \label{outletBoundary1}
\end{align}
In addition, the shock position satisfies the ODE
\begin{equation}\label{shockposition}
	\sigma^{\prime}(t) = \mathcal{F}\left(\sigma(t), \breve{\mathcal{R}}_{1}(t,\sigma(t)), \frac{1}{\alpha_{2}}\breve{\mathcal{R}}_{2}(t,\sigma(t)), \overline{\rho}_{l}(t,\sigma(t)), \overline{u}_{l}(t,\sigma(t))\right).
\end{equation}
\section{Existence of Time--periodic Transonic Shock Solutions}\label{s3}
In this section, we prove our main \Cref{t1}.
\subsection{The approximation scheme}
\par For the perturbation on the left of the shock $\left(\overline{\rho}_{l}, \overline{u}_{l}\right)$, given by \Cref{existSuper}, we denote
$$\overline{\rho}_{l}(t,x) = \rho_{-}^{(\mathcal{T})}(t,x)-\tilde{\rho}_-(x),\quad \overline{u}_{l}(t,x) = u_{-}^{(\mathcal{T})}(t,x)-\tilde{u}_-(x),\quad \overline{E}_{l}(t,x) = E_{-}^{(\mathcal{T})}(t,x)-\tilde{E}_-(x).$$
It holds that
\begin{align}
	\|\overline{\rho}_{l}\|_{C^{1}}+\|\overline{u}_{l}\|_{C^{1}}
	<C_{l}\epsilon. \label{c10}
\end{align}
\par We can then begin by building a fractional--step iteration scheme to construct a sequence of approximate solutions in the subsonic region. To set up, let
\begin{align}\label{start}
\breve{\bm{\mathcal{R}}}^{(0)}(t,x)\equiv\mathbf{0}:=(0,0)^{\top},\quad \overline{E}^{(0)}(t,x) \equiv 0,  \quad\sigma^{(0)}(t)\equiv\tilde{x}.
\end{align}
For any $k\in\mathbb{N}_{+}$, given $\left(\breve{\bm{\mathcal{R}}}^{(k-1)}, \sigma^{(k-1)}(t)\right)$, we first construct $\breve{\bm{\mathcal{R}}}^{(k)}$ as the solution to the linearized system
\begin{align}
	\partial_{t}\breve{\mathcal{R}}_{1}^{(k)}+\lambda_{1}^{(k-1)}\partial_{x}\breve{\mathcal{R}}_{1}^{(k)} &= \kappa \overline{E}^{(k-1)} + g_{1}(x, \breve{\bm{\mathcal{R}}}^{(k-1)}),\label{liearSys1}\\
	\partial_{t}\breve{\mathcal{R}}_{2}^{(k)}+\lambda_{2}^{(k-1)}\partial_{x}\breve{\mathcal{R}}_{2}^{(k)} &= \alpha_{2}\kappa \overline{E}^{(k-1)}+ \alpha_{2}g_{2}(x, \breve{\bm{\mathcal{R}}}^{(k-1)}), \label{liearSys2}
\end{align}
where the approximate electric field is given by
\begin{equation}\label{appEle}
	\overline{E}^{(k-1)}(t,x) = \int_{0}^{\sigma^{(k-1)}(t)}\left(\rho_{l}(t, y) - \tilde{\rho}(y)\right)dy + \int_{\sigma^{(k-1)}(t)}^{x}\left(\rho(y, \breve{\bm{\mathcal{R}}}^{(k-1)}(t,y)) - \tilde{\rho}(y)\right)dy,
\end{equation}
with the linearized boundary condition
\begin{align}
\breve{\mathcal{R}}_{2}^{(k)}(t,\sigma^{(k-1)}(t))=\mathscr{A}^{(k-1)}\left[t,\sigma^{(k-1)}(t)\right],\quad \breve{\mathcal{R}}_{1}^{(k)}(t,L)= \frac{1}{\alpha_{2}}\breve{\mathcal{R}}_{2}^{(k-1)}\left(t, L\right) + \overline{\mathcal{R}}_{1,b+}(t).\label{c9}
\end{align}
For simplicity, here we use the notation
\begin{equation*}
	\lambda_{i}^{(k-1)} := \lambda_{i}(x, \breve{\bm{\mathcal{R}}}^{(k-1)}), \quad i=1,2,
\end{equation*}
and
\begin{equation*}
	\mathscr{A}^{(k-1)}\left[t,\sigma^{(k-1)}(t)\right] := \alpha_{2}\mathscr{A}\left(\sigma^{(k-1)}(t), \breve{\mathcal{R}}_{1}^{(k-1)}(t,\sigma^{(k-1)}(t)), \overline{\rho}_{l}(t,\sigma^{(k-1)}(t)), \overline{u}_{l}(t,\sigma^{(k-1)}(t))\right).
\end{equation*}
We also use notations
\begin{equation*}
	\frac{\partial \mathscr{A}^{(k-1)}}{\partial x}\left[t,\sigma^{(k-1)}(t)\right], \frac{\partial \mathscr{A}^{(k-1)}}{\partial \overline{\mathcal{R}}_{1}}\left[t,\sigma^{(k-1)}(t)\right], \frac{\partial \mathscr{A}^{(k-1)}}{\partial \overline{\rho}}\left[t,\sigma^{(k-1)}(t)\right], \frac{\partial \mathscr{A}^{(k-1)}}{\partial \overline{u}}\left[t,\sigma^{(k-1)}(t)\right]
\end{equation*}
to denote the corresponding derivatives on the shock boundary. 
\par We claim that for any $k \in \mathbb{N}_{+}$, the solution to \eqref{liearSys1}--\eqref{c9} exists in a uniform existence region $(t,x)\in\mathbb{R}\times[\tilde{x}-\delta,L]$ for sufficiently small $\delta$. In fact, first, one can extend the boundary conditions periodically to $t\in \mathbb{R}$. Then, for linearized systems and decoupled boundary conditions, one can easily write the classical solution by method of characteristics. Additionally, the characteristics emanating from $\sigma^{(k-1)}(t)$ can develop both forward and backward with respect to $x$.
\par Second, with such $\breve{\bm{\mathcal{R}}}^{(k)}$, by \Cref{lemmaODE1} and \eqref{pFpx}, we construct the approximate shock position $\sigma^{(k)}(t)$ as the $\mathcal{T}$--periodic solution of the ODE
\begin{align}\label{c11}
\frac{d\sigma^{(k)}(t)}{dt}=\mathcal{F}^{(k)}\left[t,\sigma^{(k)}(t)\right],
\end{align}
where we use the notation
\begin{equation*}
\mathcal{F}^{(k)}\left[t,\sigma^{(k)}(t)\right]
:= \mathcal{F}\left(\sigma^{(k)}(t), \left(\breve{\mathcal{R}}_{1}^{(k)}, \frac{1}{\alpha_{2}}\breve{\mathcal{R}}_{2}^{(k)}\right)(t,\sigma^{(k)}(t)), \overline{\rho}_{l}(t,\sigma^{(k)}(t)), \overline{u}_{l}(t,\sigma^{(k)}(t))\right).
\end{equation*}
\par The remainder of this section is devoted to proving the following proposition.
\begin{proposition}\label{propo}
	For the iteration scheme \eqref{liearSys1}-- \eqref{c11} with \eqref{boundarySmall}--\eqref{boundaryPeriodicity} and \eqref{c10}, under the assumptions~ \eqref{restore} and \eqref{a11}--\eqref{AA11}, there exist $\theta\in (0,1)$, sufficiently small $\epsilon_1>0$ and sufficiently large $C_{P}, C_{\mathcal{F}},C_{M}>0$ such that for any given $\epsilon\in(0,\epsilon_1)$, the approximate sequence $\breve{\bm{\mathcal{R}}}^{(k)}, \sigma^{(k)}$ satisfies that for any $k\in\mathbb{N}_{+}$:
	\begin{align}
		&\breve{\bm{\mathcal{R}}}^{(k)}(t+\mathcal{T},x)=\breve{\bm{\mathcal{R}}}^{(k)}(t,x),
		~~\sigma^{(k)}(t+\mathcal{T})=\sigma^{(k)}(t), \quad \forall (t,x)\in \mathbb{R}\times[\tilde{x}-\delta,L],\label{appPeriodic}\\
		&\|\breve{\bm{\mathcal{R}}}^{(k)}\|_{C^{1}}:= \max_{i=1,2}\left\{\|\breve{\mathcal{R}}_{i}^{(k)}\|, \|\partial_{t}\breve{\mathcal{R}}_{i}^{(k)}\|, \|\partial_{x}\breve{\mathcal{R}}_{i}^{(k)}\| \right\}<C_{P}\epsilon,\label{C1estima}\\
		&\|\sigma^{(k)}-\tilde{x}\|_{C^{1}} := \max\left\{\|\sigma^{(k)}-\tilde{x}\|, \|{\sigma^{(k)}}^{\prime}\|\right\}<C_{P}\epsilon\kappa^{-1},\label{appShockesti} \\
		&\|\breve{\bm{\mathcal{R}}}^{(k)}-\breve{\bm{\mathcal{R}}}^{(k-1)}\|\leq C_{P}\epsilon\theta^{k-1},\label{cauchyEstima1}\\
		&\|\sigma^{(k)}-\sigma^{(k-1)}\|_{C^{1}}\leq C_{\mathcal{F}}\epsilon\kappa^{-1}\theta^{k-1},\label{cauchyEstima2}
	\end{align}
	and
	\begin{align}
		\max\limits_{i=1,2}\left\{\varpi(\eta|\partial_{t}\breve{\mathcal{R}}_{i}^{(k)}(\cdot,\cdot))
		+\varpi(\eta|\partial_{x}\breve{\mathcal{R}}_{i}^{(k)}(\cdot,\cdot))\right\}<C_{M}\Phi(\eta). \label{modulus}
	\end{align}
	Here $\varpi$ is the modulus of continuity defined by
	\begin{equation*}
		\varpi(\eta|h(\cdot,\cdot)):=\sup\limits_{\substack{|t_{1}-t_{2}|\leq\eta\\|x_{1}-x_{2}|\leq\eta }}
		\left|h(t_{1},x_{1})-h(t_{2},x_{2})\right|,
	\end{equation*}
	for some function $h\in C_{t,x}^{0}$ and $\Phi(\eta)$ is a continuous function on $\eta \in (0,1)$, independent of $k$, to be determined later with
	\begin{equation*}
		\lim\limits_{\eta\to 0^{+}}\Phi(\eta)=0.
	\end{equation*}
\end{proposition}
Similarly as \cite{Qu}, once we show \Cref{propo}, one can directly see that \Cref{t1} holds. In fact, first, by \eqref{cauchyEstima1} and \eqref{cauchyEstima2}, the sequence $\{\breve{\bm{\mathcal{R}}}^{(k)}\}_{k=1}^{\infty}$ and $\{\sigma^{(k)}\}_{k=1}^{\infty}$ are Cauchy sequences in $C^{0}$ and $C^{1}$ spaces respectively, and thus converge to some $C^{0}$ function $\breve{\bm{\mathcal{R}}}^{(\mathcal{T})}$ and  $C^{1}$ function $\sigma^{(\mathcal{T})}$, respectively. By \eqref{appPeriodic}, the limits are $\mathcal{T}$--periodic. Moreover, by \eqref{C1estima} and \eqref{modulus}, applying the Arzel\`{a}--Ascoli theorem, there exists a subsequence of $\{\breve{\bm{\mathcal{R}}}^{(k)}\}$ that converges uniformly in $C^{1}$. By the uniqueness of the limit, one can conclude that $\{\breve{\bm{\mathcal{R}}}^{(k)}\}$ converges to $\breve{\bm{\mathcal{R}}}^{(\mathcal{T})}$ in $C^{1}$ space. Thus, with the shock position $\sigma^{(\mathcal{T})}$, one can see that $\breve{\bm{\mathcal{R}}}^{(\mathcal{T})}|_{\Omega_{+}^{\mathcal{T}}}$ is a $C^{1}$ classical solution to the system \eqref{sys57}--\eqref{shockposition} in the subsonic region defined by \eqref{regions}. Applying the variable substitution inversely, one obtains the solution of \eqref{a1} and \eqref{a5} in the subsonic region. As for the solution in the supersonic region, we take $(\rho_{-}^{(\mathcal{T})},u_{-}^{(\mathcal{T})},E_{-}^{(\mathcal{T})})|_{\Omega_{-}^{\mathcal{T}}}$, which is given by \Cref{existSuper}. The electric field is directly recovered by \eqref{barEle}. Moreover, one can recover the R--H conditions, equivalently \eqref{RHcon} and \eqref{speedEx}, by passing to the limit in \eqref{c9}, \eqref{c11}. And the Lax entropy condition follows. Finally, we take the initial data just in the form of \eqref{a14}, and then the proof of \Cref{t1} is complete.
\par The proof of \Cref{propo} is quite lengthy. We will divide it into three subsections.
\subsection{Uniform \texorpdfstring{$C^1$}{C1} boundedness and time--periodicity}\label{suu1}
\begin{proof}[Proof of \eqref{appPeriodic}--\eqref{appShockesti}]
\par For $k\ge 1$, with $\breve{\bm{\mathcal{R}}}^{(k-1)}$ and $\sigma^{(k-1)}$ in hand, we will inductively prove the following
\begin{align}
	&\breve{\bm{\mathcal{R}}}^{(k)}(t+\mathcal{T},x)=\breve{\bm{\mathcal{R}}}^{(k)}(t,x),
	~~\sigma^{(k)}(t+\mathcal{T})=\sigma^{(k)}(t), \quad \forall (t,x)\in \mathbb{R}\times[\tilde{x}-\delta,L], \label{c17}\\
	&\|\breve{\bm{\mathcal{R}}}^{(k)}\|<C_{1}\epsilon,\quad \|\partial_{t}\breve{\bm{\mathcal{R}}}^{(k)}\|<C_{1}\epsilon, \quad \|\partial_{x}\breve{\bm{\mathcal{R}}}^{(k)}\|<C_{2}\epsilon,\label{c18}\\
	&\|\sigma^{(k)}-\tilde{x}\|<C_{\mathcal{F},0}\epsilon\kappa^{-1},\quad \|{\sigma^{(k)}}^{\prime}\|<C_{\mathcal{F},1}\epsilon,\label{c19}
\end{align}
under the assumption 
\begin{align}
	&\breve{\bm{\mathcal{R}}}^{(k-1)}(t+\mathcal{T},x)=\breve{\bm{\mathcal{R}}}^{(k-1)}(t,x),
	~~\sigma^{(k-1)}(t+\mathcal{T})=\sigma^{(k-1)}(t), \quad \forall (t,x)\in \mathbb{R}\times[\tilde{x}-\delta,L], \label{c20}\\
	&\|\breve{\bm{\mathcal{R}}}^{(k-1)}\|<C_{1}\epsilon,\quad \|\partial_{t}\breve{\bm{\mathcal{R}}}^{(k-1)}\|<C_{1}\epsilon, \quad \|\partial_{x}\breve{\bm{\mathcal{R}}}^{(k-1)}\|<C_{2}\epsilon,\label{c21}\\
	&\|\sigma^{(k-1)}-\tilde{x}\|<C_{\mathcal{F},0}\epsilon\kappa^{-1},\quad \|{\sigma^{(k-1)}}^{\prime}\|<C_{\mathcal{F},1}\epsilon,\label{c22}
\end{align}
where $C_{1},C_{2}, C_{\mathcal{F},0},C_{\mathcal{F},1}$ are positive constants to be determined later. In view of \eqref{c21}, we have $\breve{\bm{\mathcal{R}}}^{(k-1)} \in \mathcal{U}$ for some uniform small neighborhood $\mathcal{U}$, which guarantees that we only need to consider small perturbations. To start with, for $k=1$, \eqref{c20}--\eqref{c22} follow directly from \eqref{start}. 
\par First of all, by \eqref{c20}, one can easily see that $\breve{\bm{\mathcal{R}}}^{(k)}(t+\mathcal{T},x)$ solves equations \eqref{liearSys1}--\eqref{c9} if $\breve{\bm{\mathcal{R}}}^{(k)}(t,x)$ does. Thus, by uniqueness for this linear system, one has 
\begin{equation*}
	\breve{\bm{\mathcal{R}}}^{(k)}(t+\mathcal{T},x)=\breve{\bm{\mathcal{R}}}^{(k)}(t,x),
	\quad \forall (t,x)\in \mathbb{R}\times[\tilde{x}-\delta,L].
\end{equation*}
\par Then, we present the estimates of the 2--wave $\breve{\mathcal{R}}_{2}^{(k)}$ in detail, which is the most essential part. For any fixed $(t,x) \in \mathbb{R}\times[\tilde{x}-\delta, L]$, one can define the 2--characteristics as
\begin{align}\label{characteritic}
\left\{
\begin{aligned}
&\frac{dX_{2}^{(k)}}{d\tau}(\tau;t,x)=\lambda_{2}\left(X_{2}^{(k)}(\tau;t,x),\breve{\bm{\mathcal{R}}}^{(k-1)}\left(\tau,X_{2}^{(k)}(\tau;t,x)\right)\right),\\
&X_{2}^{(k)}(t;t,x)=x,
\end{aligned}
\right.
\end{align}
and set the backward exit--time for $(t,x)$ as $\tau_{2}^{(k)}(t,x)$ by the intersection of the characteristic $X_{2}^{(k)}$ and the shock $\sigma^{(k-1)}$, i.e.,
\begin{equation}\label{exitTime}
	X_{2}^{(k)}(\tau_{2}^{(k)};t,x)=\sigma^{(k-1)}(\tau_{2}^{(k)}).
\end{equation}
Due to the different propagation speeds, the intersection is uniquely well--defined. One can integrate \eqref{liearSys2} along the characteristic curve $X_{2}^{(k)}=X_{2}^{(k)}(\tau;t,x)$ to get
\begin{multline}
	\breve{\mathcal{R}}_{2}^{(k)}(t,x)=\breve{\mathcal{R}}_{2}^{(k)}\left(\tau_{2}^{(k)},\sigma^{(k-1)}(\tau_{2}^{(k)})\right) \\
	+\int_{\tau_{2}^{(k)}}^{t}\left(\alpha_{2}\kappa \overline{E}^{(k-1)} +\alpha_{2}g_{2}(x, \breve{\bm{\mathcal{R}}}^{(k-1)})\right)\left(\tau,X_{2}^{(k)}(\tau;t,x)\right)d\tau.
\end{multline}
From \eqref{Ex} and \eqref{appEle}, one can see
\begin{equation*}
	\|\overline{E}^{(k-1)}\| \le \|\dfrac{\tilde{\rho}}{\tilde{c}}\|LC_{1}\epsilon + C\epsilon^{2}.
\end{equation*}
And then by \eqref{a11}, \eqref{c9} and applying Hadamard's formula, one has
\begin{align}
\left|\breve{\mathcal{R}}_{2}^{(k)}(t,x)\right|
\leq&\left|\mathscr{A}^{(k-1)}\left[\tau_{2}^{(k)},\sigma^{(k-1)}(\tau_{2}^{(k)})\right]\right|
+C\kappa\epsilon +CC_{1}\epsilon^{2} \notag\\
\leq& \alpha_{2}(1+C\epsilon +C\kappa)\left(|\frac{\partial\mathscr{A}}{\partial x}(\tilde{x},\bm{0})|\|{\sigma^{(k-1)}}-\tilde{x}\| +|\frac{\partial\mathscr{A}}{\partial \overline{\mathcal{R}}_{1}}(\tilde{x},\bm{0})|\|\overline{\mathcal{R}}_{1}^{(k-1)}\| \right.\notag\\
&\left. +|\frac{\partial\mathscr{A}}{\partial \overline{\rho}_{l}}(\tilde{x},\bm{0})|\|\overline{\rho}_{l}\| +|\frac{\partial\mathscr{A}}{\partial \overline{u}_{l}}(\tilde{x},\bm{0})|\|\overline{u}_{l}\| \right)
+C\kappa\epsilon +C\epsilon^{2}.\label{c36}
\end{align}
Here and below $C$ denotes different constants independent of $k$. Furthermore, applying \Cref{lemmaODE1}, one has
\begin{multline}\label{chiEsti}
	\|\sigma^{(k-1)}-\tilde{x}\|\leq(1+C\epsilon)\frac{1}{|\frac{\partial \mathcal{F}}{\partial x}(\tilde{x},\bm{0})|}\left(\sum_{i=1}^{2}\dfrac{1}{\alpha_{i}}|\frac{\partial \mathcal{F}}{\partial \overline{\mathcal{R}}_{i}}(\tilde{x},\bm{0})|\|\breve{\mathcal{R}}_{i}^{(k-1)}\| \right. \\
	\left. +|\frac{\partial \mathcal{F}}{\partial \overline{\rho}_{l}}(\tilde{x},\bm{0})|\|\overline{\rho}_{l}\|
	+|\frac{\partial \mathcal{F}}{\partial \overline{u}_{l}}(\tilde{x},\bm{0})|\|\overline{u}_{l}\|\right).
\end{multline}
Therefore, applying the structural dissipation \eqref{dissipationCon}, one can get
\begin{equation}\label{C0estiFor3}
	\left|\breve{\mathcal{R}}_{2}^{(k)}(t,x)\right|
	\leq (1+C\epsilon +C\kappa)\alpha_{2}\theta_{d}C_{1}\epsilon + CC_{l}\epsilon +C\kappa\epsilon +C\epsilon^{2}.
\end{equation}
Thus, by \eqref{beta}, for large enough $C_{1}$, one has
\begin{equation}
	\|\breve{\mathcal{R}}_{2}^{(k)}\| \le C_{1}\epsilon.
\end{equation}
The dissipation for $1$--waves at the outlet boundary is guaranteed by $\alpha_{2} >1$. The analysis is similar, so we omit here and in subsequent part of this paper.
\par Next, we turn to the estimates for derivatives. We need to estimate the derivative in the region using the derivative along the boundary $\sigma^{(k-1)}(t)$. To this end, we define function $\omega^{(k)}_{2}(t,x)$ as
\begin{equation}
	\omega_{2}^{(k)}\left(t,x\right)=\mathcal{F}^{(k-1)}\left(\tau_{2}^{(k)}(t,x),\sigma^{(k-1)}(\tau_{2}^{(k)}(t,x))\right),
\end{equation}
which is well--defined following the definition of the backward exit--time. One can easily see that $\omega_{2}^{(k)}$ remains invariant along the characteristics, namely, it satisfies
\begin{equation}\label{omegaInvariant}
	\frac{\partial \omega_{2}^{(k)}}{\partial t} + \lambda_{2}^{(k-1)}\frac{\partial \omega_{2}^{(k)}}{\partial x} = 0.
\end{equation} Moreover, differentiating the boundary condition \eqref{c9} with respect to $t$ yields
\begin{equation}\label{tDerivation}
	\frac{d}{dt}\breve{\mathcal{R}}_{2}^{(k)}(t,\sigma^{(k-1)}(t)) = \left(\frac{\partial}{\partial t}+\mathcal{F}^{(k-1)}(t,\sigma^{(k-1)}(t))\frac{\partial}{\partial x}\right)\breve{\mathcal{R}}_{2}^{(k)}(t,\sigma^{(k-1)}(t)).
\end{equation}
We define the differential operator 
\begin{equation}\label{dOmega}
	D_{\omega}^{(k)}:=\frac{\partial}{\partial t}+\omega_{2}^{(k)}(t,x)\frac{\partial}{\partial x}
\end{equation}
and apply it to the equation \eqref{liearSys2}
\begin{align}
	&\left(\partial_{t}+\lambda_{2}^{(k-1)}\partial_{x}\right)D_{\omega}^{(k)}\breve{\mathcal{R}}_{2}^{(k)}\notag\\
	=&-D_{2}^{(k)}\lambda_{2}^{(k-1)}\partial_{x}\breve{\mathcal{R}}_{2}^{(k)} + \kappa\alpha_{2}D_{\omega}^{(k)}\overline{E}^{(k-1)} + \alpha_{2}D_{2}^{(k)}g_{2} \label{deriTrans}
\end{align}
where we use the operator notation for some function $f=f(x,\breve{\bm{\mathcal{R}}}^{(k-1)})$:
\begin{equation}\label{D3Forshort}
	D_{2}^{(k)} := \sum_{i=1}^{2}\left(\frac{\partial\breve{\mathcal{R}}_{i}^{(k-1)}}{\partial t}\frac{\partial}{\partial \breve{\mathcal{R}}_{i}} + \omega_{2}^{(k)}\frac{\partial\breve{\mathcal{R}}_{i}^{(k-1)}}{\partial x}\frac{\partial}{\partial \breve{\mathcal{R}}_{i}}\right) + \omega_{2}^{(k)}\frac{\partial}{\partial x}
\end{equation}
for simplicity. Then, one can integrate \eqref{deriTrans} along the characteristic curve $X_{2}^{(k)}=X_{2}^{(k)}(\tau;t,x)$ and use \eqref{appEle}, \eqref{tDerivation} and \eqref{c21}--\eqref{c22} to get
\begin{align}
	&\left|D_{\omega}^{(k)}\breve{\mathcal{R}}_{2}^{(k)}(t,x)\right| \notag\\
	\le& \left|D_{\omega}^{(k)}\breve{\mathcal{R}}_{2}^{(k)}\left(\tau_{2}^{(k)},\sigma^{(k-1)}(\tau_{2}^{(k)})\right)\right| +C\epsilon\|\partial_{x}\breve{\mathcal{R}}_{2}^{(k)}\|+C\epsilon^{2}+C\kappa\epsilon\notag\\
	\le& \alpha_{2}(1+C\epsilon +C\kappa)\left(|\frac{\partial\mathscr{A}}{\partial x}(\tilde{x},\bm{0})|\|{\sigma^{(k-1)}}^{\prime}\| +|\frac{\partial\mathscr{A}}{\partial \overline{\mathcal{R}}_{1}}(\tilde{x},\bm{0})|\left(\|\partial_{t}\overline{\mathcal{R}}_{1}^{(k-1)}\|+\|{\sigma^{(k-1)}}^{\prime}\|\|\partial_{x}\overline{\mathcal{R}}_{1}^{(k-1)}\|\right) \right)\notag\\
	&+CC_{l}\epsilon +C\epsilon\|\partial_{x}\breve{\mathcal{R}}_{2}^{(k)}\|
	+C\kappa\epsilon +C\epsilon^{2}. \label{c42}
\end{align}
Again applying \Cref{lemmaODE1}, one can get
\begin{equation}\label{chiPrimeEsti}
	\|{\sigma^{(k-1)}}^{\prime}\|\leq
	(2+C\epsilon)\left(\sum_{i=1}^{2}\dfrac{1}{\alpha_{i}}|\frac{\partial \mathcal{F}}{\partial \overline{\mathcal{R}}_{i}}(\tilde{x},\bm{0})|\|\breve{\mathcal{R}}_{i}^{(k-1)}\| +|\frac{\partial \mathcal{F}}{\partial \overline{\rho}_{l}}(\tilde{x},\bm{0})|\|\overline{\rho}_{l}\|
	+|\frac{\partial \mathcal{F}}{\partial \overline{u}_{l}}(\tilde{x},\bm{0})|\|\overline{u}_{l}\|\right).
\end{equation}
Thus, combining \eqref{pA3px},\eqref{c10}, \eqref{c21} and \eqref{c22}, one has
\begin{equation}\label{c43}
	\|\partial_{t}\breve{\mathcal{R}}_{2}^{(k)}\| - C\epsilon\|\partial_{x}\breve{\mathcal{R}}_{2}^{(k)}\| \le \alpha_{2}(1+C\epsilon +C\kappa)\left(C\kappa\epsilon+\theta_{d}C_{1}\epsilon\right)+CC_{l}\epsilon 
	+C\kappa\epsilon +C\epsilon^{2}.
\end{equation}
Using the linear equation \eqref{liearSys2}, one can directly see
\begin{align}
\|\partial_{x}\breve{\mathcal{R}}_{2}^{(k)}\|\leq\mu_{\max}\|\partial_{t}\breve{\mathcal{R}}_{2}^{(k)}\|+C\kappa\epsilon+C\epsilon^{2},\label{c44}
\end{align}
where
\begin{equation}\label{mumax}
	\mu_{\max}=\max\limits_{\substack{i=1,2\\\breve{\bm{\mathcal{R}}}\in \mathcal{U}}}\frac{1}{\left|\lambda_{i}\left(x,\breve{\bm{\mathcal{R}}}\right)\right|}.
\end{equation}
Combining \eqref{c43} and \eqref{c44}, one has
\begin{align}
&\|\partial_{t}\breve{\mathcal{R}}_{2}^{(k)}\|<C_{1}\epsilon,\quad \|\partial_{x}\breve{\mathcal{R}}_{2}^{(k)}\|<C_{2}\epsilon.\label{c46}
\end{align}
for large enough $C_{1}$ and $C_{2}$.
\par Finally, with \eqref{c18} in hand, one can see that \eqref{c19} follows just as in \eqref{chiEsti} and \eqref{chiPrimeEsti}, with the index replaced by $k$. Furthermore, the $\mathcal{T}$--periodicity of $\sigma^{(k)}$ is guaranteed by \Cref{lemmaODE1}. And so far, we have proved \eqref{c17}--\eqref{c19}, and thus \eqref{appPeriodic}--\eqref{appShockesti}.
\end{proof}

\subsection{Cauchy properties}\label{suu2}
\begin{proof}[Proof of \eqref{cauchyEstima1} and \eqref{cauchyEstima2}]
For $k \ge 2$, we will inductively prove
\begin{align}
	&\|\breve{\bm{\mathcal{R}}}^{(k)}-\breve{\bm{\mathcal{R}}}^{(k-1)}\|\leq C_{1}\epsilon\theta^{k-1}, \label{c52}\\
	&\|\sigma^{(k)}-\sigma^{(k-1)}\|\leq C_{\mathcal{F}}\epsilon\kappa^{-1}\theta^{k-1}, \quad \|{\sigma^{(k)}}^{\prime}-{\sigma^{(k-1)}}^{\prime}\| \leq C_{\mathcal{F}}\epsilon\theta^{k-1},\label{c52.1}
\end{align}
under the assumption \eqref{appPeriodic}--\eqref{appShockesti} and
\begin{align}
&\|\breve{\bm{\mathcal{R}}}^{(k-1)}-\breve{\bm{\mathcal{R}}}^{(k-2)}\|\leq C_{1}\epsilon\theta^{k-2}, \label{c53}\\
&\|\sigma^{(k-1)}-\sigma^{(k-2)}\|\leq C_{\mathcal{F}}\epsilon\kappa^{-1}\theta^{k-2}, \quad \|{\sigma^{(k-1)}}^{\prime}-{\sigma^{(k-2)}}^{\prime}\|\leq C_{\mathcal{F}}\epsilon\theta^{k-2},\label{c53.1}
\end{align}
for some $\theta \in (0,1)$ and $C_{\mathcal{F}}>0$ to be determined later. To start with, for $k=2$, \eqref{c53}--\eqref{c53.1} follows directly from \eqref{start} and \eqref{c18}--\eqref{c19}.
\par For $k \ge 2$, the equation \eqref{liearSys2} yields
\begin{align}
	&\left(\partial_{t}+\lambda_{2}^{(k-1)}\partial_{x}\right)\left(\breve{\mathcal{R}}_{2}^{(k)}- \breve{\mathcal{R}}_{2}^{(k-1)}\right) \notag \\
	=& - \left(\lambda_{2}^{(k-1)}-\lambda_{2}^{(k-2)}\right)\partial_{x}\breve{\mathcal{R}}_{2}^{(k-1)} +\alpha_{2}\left(g_{2}(x,\breve{\bm{\mathcal{R}}}^{(k-1)})-g_{2}(x,\breve{\bm{\mathcal{R}}}^{(k-2)})\right) \notag \\
	&+ \kappa\int_{\sigma^{(k-2)}(t)}^{\sigma^{(k-1)}(t)}\left(\rho_{l}(t, y) - \rho(y, \breve{\bm{\mathcal{R}}}^{(k-2)}(t,y))\right)dy \notag\\
	&+ \kappa\int_{\sigma^{(k-1)}(t)}^{x}\left(\rho(y, \breve{\bm{\mathcal{R}}}^{(k-1)}(t,y)) - \rho(y, \breve{\bm{\mathcal{R}}}^{(k-2)}(t,y))\right)dy. \label{c54}
\end{align}
Then for any fixed $(t,x) \in \mathbb{R}\times[\tilde{x}-\delta, L]$, one can integrate \eqref{c54} along the characteristic $X_{2}^{(k)}$ from $\tau_{2}^{(k)}(t,x)$ to $t$ and apply Hadamard's formula to get
\begin{align}
	&\left|\breve{\mathcal{R}}_{2}^{(k)}(t,x)-\breve{\mathcal{R}}_{2}^{(k-1)}(t,x)\right| \notag \\
	\le & \left|\breve{\mathcal{R}}_{2}^{(k)}\left(\tau_{2}^{(k)},\sigma^{(k-1)}(\tau_{2}^{(k)})\right)
	-\breve{\mathcal{R}}_{2}^{(k-1)}\left(\tau_{2}^{(k)},\sigma^{(k-2)}(\tau_{2}^{(k)})\right)\right|  \notag \\
	&+ \left|\breve{\mathcal{R}}_{2}^{(k-1)}\left(\tau_{2}^{(k)},\sigma^{(k-2)}(\tau_{2}^{(k)})\right)
	-\breve{\mathcal{R}}_{2}^{(k-1)}\left(\tau_{2}^{(k)},\sigma^{(k-1)}(\tau_{2}^{(k)})\right)\right| + \left(C\epsilon+C\kappa\right)\|\breve{\bm{\mathcal{R}}}^{(k-1)}-\breve{\bm{\mathcal{R}}}^{(k-2)}\| \notag \\
	\le & \alpha_{2}(1+C\epsilon +C\kappa)\left(|\frac{\partial\mathscr{A}}{\partial x}(\tilde{x},\bm{0})|\|\sigma^{(k-1)}-\sigma^{(k-2)}\| +|\frac{\partial\mathscr{A}}{\partial \overline{\mathcal{R}}_{1}}(\tilde{x},\bm{0})|\|\overline{\mathcal{R}}_{1}^{(k-1)}-\overline{\mathcal{R}}_{1}^{(k-2)}\| \right)\notag\\
	&+ (CC_{l}+ C_{1})\epsilon\|\sigma^{(k-1)}-\sigma^{(k-2)}\| +\left(C\epsilon+C\kappa\right)\|\breve{\bm{\mathcal{R}}}^{(k-1)}-\breve{\bm{\mathcal{R}}}^{(k-2)}\|. \label{c55}
\end{align}
Applying \Cref{lemmaODE2} to equation \eqref{c11} with number $k-1$ and $k-2$, one can get
\begin{multline}\label{c56}
	\|\sigma^{(k-1)}-\sigma^{(k-2)}\|\leq(1+C\epsilon)\frac{\exp(|\frac{\partial \mathcal{F}}{\partial x}(\tilde{x},\bm{0})|\mathcal{T})}{|\frac{\partial \mathcal{F}}{\partial x}(\tilde{x},\bm{0})|}\left(\sum_{i=1}^{2}\dfrac{1}{\alpha_{i}}|\frac{\partial \mathcal{F}}{\partial \overline{\mathcal{R}}_{i}}(\tilde{x},\bm{0})|\|\breve{\mathcal{R}}_{i}^{(k-1)}-\breve{\mathcal{R}}_{i}^{(k-2)}\| \right. \\
	\left. +\left(|\frac{\partial \mathcal{F}}{\partial \overline{\rho}_{l}}(\tilde{x},\bm{0})|
	+|\frac{\partial \mathcal{F}}{\partial \overline{u}_{l}}(\tilde{x},\bm{0})|\right)C_{l}\epsilon\|\sigma^{(k-1)}-\sigma^{(k-2)}\|\right).
\end{multline}
By \eqref{pFpx}, one has
\begin{equation*}
	\exp\left(\left|\frac{\partial \mathcal{F}}{\partial x}(\tilde{x},\bm{0})\right|\mathcal{T}\right) = 1+O(1)\kappa,
\end{equation*}
and thus, combining \eqref{c55}, \eqref{c56} and \eqref{dissipationCon},
\begin{multline}\label{cauchy}
	\left|\breve{\mathcal{R}}_{2}^{(k)}(t,x)-\breve{\mathcal{R}}_{2}^{(k-1)}(t,x)\right| \\
	\le (1+C\epsilon +C\kappa +\frac{C\epsilon}{\kappa})\alpha_{2}\theta_{d} \|\breve{\bm{\mathcal{R}}}^{(k-1)}-\breve{\bm{\mathcal{R}}}^{(k-2)}\| +\left(C\epsilon+C\kappa\right)\|\breve{\bm{\mathcal{R}}}^{(k-1)}-\breve{\bm{\mathcal{R}}}^{(k-2)}\|.
\end{multline}
For small enough $\epsilon$, one can find suitable $\theta$ such that
\begin{equation*}
	(1+C\epsilon +C\kappa+\frac{C\epsilon}{\kappa})\alpha_{2}\theta_{d} + C\epsilon +C \kappa < \theta < 1.
\end{equation*}
Then with \eqref{c53} in hand, one has
\begin{align}
\left|\breve{\mathcal{R}}_{2}^{(k)}(t,x)-\breve{\mathcal{R}}_{2}^{(k-1)}(t,x)\right|\leq C_{1}\epsilon\theta^{k-1}.\label{c61}
\end{align}
The analysis for $\breve{\mathcal{R}}_{1}^{(k)}$ is similar. With \eqref{c52} in hand, \eqref{c52.1} follows for sufficiently large $C_{\mathcal{F}}$ by applying \Cref{lemmaODE2}. This completes the proof of \eqref{cauchyEstima1}--\eqref{cauchyEstima2}.
\end{proof}

\subsection{Estimates of continuity modulus}\label{suu3}
\begin{proof}[Proof of \eqref{modulus}]
To estimate the modulus of continuity as in \eqref{modulus}, we will prove for $k\ge 1$,
\begin{align}
&\mathop{\max}\limits_{i=1,2}\varpi(\eta|\partial_{t}\breve{\mathcal{R}}_{i}^{(k)}(\cdot,x))<C_{M,1}\Phi(\eta), \quad \mathop{\max}\limits_{i=1,2}\varpi(\eta|\partial_{x}\breve{\mathcal{R}}_{i}^{(k)}(\cdot,x))<C_{M,2}\Phi(\eta),\label{c63}\\
&\max\limits_{i=1,2}\left\{\varpi(\eta|\partial_{t}\breve{\mathcal{R}}_{i}^{(k)}(\cdot,\cdot))
+\varpi(\eta|\partial_{x}\breve{\mathcal{R}}_{i}^{(k)}(\cdot,\cdot))\right\}<C_{M}\Phi(\eta),\label{c64}
\end{align}
under the assumption \eqref{appPeriodic}--\eqref{cauchyEstima2} and
\begin{align}
	&\mathop{\max}\limits_{i=1,2}\varpi(\eta|\partial_{t}\breve{\mathcal{R}}_{i}^{(k-1)}(\cdot,x))<C_{M,1}\Phi(\eta), \quad \mathop{\max}\limits_{i=1,2}\varpi(\eta|\partial_{x}\breve{\mathcal{R}}_{i}^{(k-1)}(\cdot,x))<C_{M,2}\Phi(\eta),\label{c65}\\
	&\max\limits_{i=1,2}\left\{\varpi(\eta|\partial_{t}\breve{\mathcal{R}}_{i}^{(k-1)}(\cdot,\cdot))
	+\varpi(\eta|\partial_{x}\breve{\mathcal{R}}_{i}^{(k-1)}(\cdot,\cdot))\right\}<C_{M}\Phi(\eta),\label{c66}
\end{align}
where
\begin{equation*}
	\varpi(\eta|h(\cdot,x))=\mathop{\sup}\limits_{|t_{1}-t_{2}|\leq\eta}|h(t_{1},x)-h(t_{2},x)|,
\end{equation*}
and $C_{M,1},C_{M,2}$ and $C_{M}$ are constants to be determined later. Moreover, we define the modulus function by
\begin{align}
\Phi(\eta)=&\eta+\varpi(\eta|\overline{\mathcal{R}}_{1,b+}^{\prime})+\varpi(\eta|\frac{\partial\overline{\rho}_{l}}{\partial t}(\cdot,\cdot))
+\varpi(\eta|\frac{\partial\overline{\rho}_{l}}{\partial x}(\cdot,\cdot))+\varpi(\eta|\frac{\partial\overline{u}_{l}}{\partial t}(\cdot,\cdot))+\varpi(\eta|\frac{\partial\overline{u}_{l}}{\partial x}(\cdot,\cdot)),\label{c67}
\end{align}
where for $h \in C_{t}^{1}$
\begin{equation*}
	\varpi(\eta|h) = \sup_{|t_1 - t_2| \leq \eta} \left|h(t_1) - h(t_2)\right|.
\end{equation*}
By \eqref{boundarySmall} and \eqref{c10}, one has $\lim_{\eta \to 0+} \Phi(\eta) = 0$. To start with, \eqref{start} yields \eqref{c65}--\eqref{c66} for $k=1$ directly. In the rest of this subsection, we will prove \eqref{c63}--\eqref{c64} in four steps.
\begin{description}[leftmargin=0cm]
	\item[Step 1. Estimates for two points on the shock boundary]  For given $t_{1},t_{2}$ with $\left|t_{1}-t_{2}\right| \le \eta$, by taking the time derivative on the boundary condition \eqref{c9}, one has
	\begin{align}
		&\frac{1}{\alpha_{2}}\left(\frac{d}{dt}\breve{\mathcal{R}}_{2}^{(k)}(t_{2},\sigma^{(k-1)}(t_{2})) -  \frac{d}{dt}\breve{\mathcal{R}}_{2}^{(k)}(t_{1},\sigma^{(k-1)}(t_{1}))\right) \notag \\
		=& \left({\sigma^{(k-1)}}^{\prime}(t_{2})- {\sigma^{(k-1)}}^{\prime}(t_{1})\right)\frac{\partial \mathscr{A}^{(k-1)}}{\partial x}\left[t_{2},\sigma^{(k-1)}(t_{2})\right] \notag\\
		&+ {\sigma^{(k-1)}}^{\prime}(t_{1})\left(\frac{\partial \mathscr{A}^{(k-1)}}{\partial x}\left[t_{2},\sigma^{(k-1)}(t_{2})\right]-\frac{\partial \mathscr{A}^{(k-1)}}{\partial x}\left[t_{1},\sigma^{(k-1)}(t_{1})\right]\right) \notag \\
		&+ \left(\frac{d}{d t}\overline{\mathcal{R}}_{1}^{(k-1)}(t_{2},\sigma^{(k-1)}(t_{2}))-\frac{d}{d t}\overline{\mathcal{R}}_{1}^{(k-1)}(t_{1},\sigma^{(k-1)}(t_{1}))\right)\frac{\partial \mathscr{A}^{(k-1)}}{\partial \overline{\mathcal{R}}_{1}}\left[t_{2},\sigma^{(k-1)}(t_{2})\right] \notag \\
		&+ \frac{d}{d t}\overline{\mathcal{R}}_{1}^{(k-1)}(t_{1},\sigma^{(k-1)}(t_{1}))\left(\frac{\partial \mathscr{A}^{(k-1)}}{\partial \overline{\mathcal{R}}_{1}}\left[t_{2},\sigma^{(k-1)}(t_{2})\right] - \frac{\partial \mathscr{A}^{(k-1)}}{\partial \overline{\mathcal{R}}_{1}}\left[t_{1},\sigma^{(k-1)}(t_{1})\right]\right) \notag \\
		& + \sum_{i=1}^{2}\left(\frac{d}{d t}\overline{U}_{i,l}^{(k-1)}(t_{2},\sigma^{(k-1)}(t_{2}))-\frac{d}{d t}\overline{U}_{i,l}^{(k-1)}(t_{1},\sigma^{(k-1)}(t_{1}))\right)\frac{\partial \mathscr{A}^{(k-1)}}{\partial \overline{U}_{i,l}}\left[t_{2},\sigma^{(k-1)}(t_{2})\right] \notag \\
		&+ \sum_{i=1}^{2} \frac{d}{d t}\overline{U}_{i,l}^{(k-1)}(t_{1},\sigma^{(k-1)}(t_{1}))\left(\frac{\partial \mathscr{A}^{(k-1)}}{\partial \overline{U}_{i,l}}\left[t_{2},\sigma^{(k-1)}(t_{2})\right] - \frac{\partial \mathscr{A}^{(k-1)}}{\partial \overline{U}_{i,l}}\left[t_{1},\sigma^{(k-1)}(t_{1})\right]\right),
	\end{align}
	where we denote $(\overline{U}_{1,l},\overline{U}_{2,l}) =(\overline{\rho}_{l}, \overline{u}_{l})$ for simplicity. By \eqref{bb21}, \eqref{C1estima}, \eqref{c19} and \eqref{c65}, one has
	\begin{multline}\label{modulusDissipation}
		\left|\frac{d}{dt}\breve{\mathcal{R}}_{2}^{(k)}(t_{2},\sigma^{(k-1)}(t_{2})) -  \frac{d}{dt}\breve{\mathcal{R}}_{2}^{(k)}(t_{1},\sigma^{(k-1)}(t_{1}))\right| \le \left(C\kappa C_{P}\epsilon\eta + CC_{P}\epsilon\eta\right)\kappa+CC_{\mathcal{F},1}\epsilon^{2}\kappa\eta \\
		+\alpha_{2}\theta_{d}\left(C_{M,1}+C_{\mathcal{F},1}\epsilon C_{M,2}\right)\Phi(\eta)+C(1+C_{l}C_{1}\epsilon^{2})\Phi(\eta).
	\end{multline}
	This equation gives the dissipation structure of the continuity modulus at the shock boundary. It will help us estimate the modulus in the form
	\begin{equation}
		\left|D_{\omega}^{(k)}\breve{\mathcal{R}}_{2}^{(k)}(t_{2},x) - D_{\omega}^{(k)}\breve{\mathcal{R}}_{2}^{(k)}(t_{1},x)\right|
	\end{equation}
	with the differential operator defined by \eqref{dOmega}.
	
	\item[Step 2. Uniform $\|\omega_{2}^{(k)}\|_{C^{1}}$ estimates] We need to estimate the $C^{1}$ norm of $\omega_{2}^{(k)} = \omega_{2}^{(k)}(t,x)$. Applying the differential operator $D_{\omega}^{(k)}$ to \eqref{omegaInvariant} yields
	\begin{equation}
		\frac{\partial}{\partial t} D_{\omega}^{(k)}\omega_{2}^{(k)} + \lambda_{2}^{(k-1)}\frac{\partial}{\partial x} D_{\omega}^{(k)}\omega_{2}^{(k)} = -D_{2}^{(k)}\lambda_{2}^{(k-1)}\frac{\partial \omega_{2}^{(k)}}{\partial x}.
	\end{equation}
	Integrating it along the characteristic \eqref{characteritic} and using \eqref{bb21}, one can obtain
	\begin{align}
		&\left|D_{\omega}^{(k)}\omega_{2}^{(k)}(t,x)\right| 
		\le \left|D_{\omega}^{(k)}\omega_{2}^{(k)}\left(\tau_{2}^{(k)},\sigma^{(k-1)}(\tau_{2}^{(k)})\right)\right| +C_{\mathcal{F},1}\epsilon\|\partial_{x}\omega_{2}^{(k)}\| \notag\\
		\le& (1+C\epsilon +C\kappa)\left(\sum_{i=1}^{2}\dfrac{1}{\alpha_{i}}|\frac{\partial \mathcal{F}}{\partial \overline{\mathcal{R}}_{i}}(\tilde{x},\bm{0})|\|\partial_{t}\breve{\mathcal{R}}_{i}^{(k-1)}\| \right) +CC_{l}\epsilon +C\epsilon\|\partial_{x}\omega_{2}^{(k)}\|
		+CC_{1}\kappa\epsilon. \label{c72}
	\end{align}
	Then, one can use an argument similar to that for \eqref{c42}--\eqref{c46} to get
	\begin{align}
		\|\frac{\partial\omega_{2}^{(k)}}{\partial x}\|\leq\mu_{\max}\|\frac{\partial\omega_{2}^{(k)}}{\partial t}\|, \label{c73}
	\end{align}
	and thus obtain from \eqref{c72} that
	\begin{equation}\label{omegaC1}
		\|\frac{\partial\omega_{2}^{(k)}}{\partial t}\|<C_{\mathcal{F},2}\epsilon,\quad \|\frac{\partial\omega_{2}^{(k)}}{\partial x}\|\leq\mu_{\max}C_{\mathcal{F},2}\epsilon
	\end{equation}
	for some suitable large constant $C_{\mathcal{F},2}$ independent of $k$.
	
	\item[Step 3. Estimates in the region by the method of characteristics] To simplify the analysis, we set $\mu_{i} = (\lambda_{i})^{-1}, \mu_{i}^{(k-1)} = (\lambda_{i}^{(k-1)})^{-1}$. For any fixed $(t,x)\in \mathbb{R}\times[\tilde{x}-\delta,L]$, we denote the inverse function of characteristics \eqref{characteritic} by $T_{2}^{(k)}(y;t,x)$, i.e.,
	\begin{equation}
		X_{2}^{(k)}\left(T_{2}^{(k)}(y;t,x);t,x\right) \equiv y.
	\end{equation}
	Obviously, it is well--defined by the implicit function theorem and driven by the ODE
	\begin{align}\label{characteriticT}
		\left\{
		\begin{aligned}
			&\frac{dT_{2}^{(k)}}{dy}(y;t,x)=\mu_{2}\left(y,\breve{\bm{\mathcal{R}}}^{(k-1)}\left(T_{2}^{(k)}(y;t,x),y\right)\right),\\
			&T_{2}^{(k)}(x;t,x)=t.
		\end{aligned}
		\right.
	\end{align}
	\par Given two points $(t_{1},x), (t_{2},x)$ in the region with $|t_{1}-t_{2}| \le \eta$, for any $y\in [\tilde{x}-\delta,L]$, one can integrate \eqref{characteriticT} from $x$ to $y$ and get
	\begin{align}
		&\left|T_{2}^{(k)}(y;t_{2},x) - T_{2}^{(k)}(y;t_{1},x)\right| \notag \\
		\le &|t_{2} - t_{1}| + |\int_{x}^{y} \mu_{2}\left(z,\breve{\bm{\mathcal{R}}}^{(k-1)}\left(T_{2}^{(k)}(z;t_{2},x),z\right)\right)- \mu_{2}\left(z,\breve{\bm{\mathcal{R}}}^{(k-1)}\left(T_{2}^{(k)}(z;t_{1},x),z\right)\right)dz| \notag \\
		\le & |t_{2} - t_{1}| + CC_{P}\epsilon|\int_{x}^{y} T_{2}^{(k)}(z;t_{2},x) - T_{2}^{(k)}(z;t_{1},x)dz|.
	\end{align}
	Combining with Gronwall's inequality, one has
	\begin{equation}\label{T1T2esitm}
		\left|T_{2}^{(k)}(y;t_{2},x) - T_{2}^{(k)}(y;t_{1},x)\right| \le (1+CC_{P}\epsilon)|t_{2} - t_{1}| \le (1+CC_{P}\epsilon)\eta, \quad \forall y\in [\tilde{x}-\delta,L].
	\end{equation}
	Recalling the exit--time $\tau_{2}^{(k)}(t,x)$, we denote exit--position by
	\begin{equation*}
		\xi_{2}^{(k)}(t,x) := X_{2}^{(k)}(\tau_{2}^{(k)};t,x),
	\end{equation*}
	and by \eqref{exitTime}, one has
	\begin{equation}\label{xiChi}
		\xi_{2}^{(k)}(t,x) = \sigma^{(k-1)}(T_{2}^{(k)}(\xi_{2}^{(k)};t,x)).
	\end{equation}
	Thus, one can see
	\begin{align}
		&\left|\xi_{2}^{(k)}(t_{1},x)- \xi_{2}^{(k)}(t_{2},x)\right| = \left|\sigma^{(k-1)}(T_{2}^{(k)}(\xi_{2}^{(k)}(t_{1},x);t_{1},x))-\sigma^{(k-1)}(T_{2}^{(k)}(\xi_{2}^{(k)}(t_{2},x);t_{2},x))\right| \notag \\
		\le & \|{\sigma^{(k-1)}}^{\prime}\|\left(|T_{2}^{(k)}(\xi_{2}^{(k)}(t_{1},x);t_{1},x)-T_{2}^{(k)}(\xi_{2}^{(k)}(t_{2},x);t_{1},x)|\right.\notag\\
		&\qquad\qquad\qquad\qquad\qquad\qquad \left.+|T_{2}^{(k)}(\xi_{2}^{(k)}(t_{2},x);t_{1},x)-T_{2}^{(k)}(\xi_{2}^{(k)}(t_{2},x);t_{2},x)|\right) \notag \\
		\le & \|{\sigma^{(k-1)}}^{\prime}\|\left(\mu_{\max}\left|\xi_{2}^{(k)}(t_{1},x)- \xi_{2}^{(k)}(t_{2},x)\right|+(1+CC_{P}\epsilon)\eta\right). \notag
	\end{align}
	Consequently, one has
	\begin{equation}\label{X1X2estim}
		\left|\xi_{2}^{(k)}(t_{1},x)- \xi_{2}^{(k)}(t_{2},x)\right| \le (1+CC_{\mathcal{F},1}\epsilon)C_{\mathcal{F},1}\epsilon\eta.
	\end{equation}
	Then, one can integrate \eqref{deriTrans} along $T_{2}^{(k)}(y;t_{1},x)$ and $T_{2}^{(k)}(y;t_{2},x)$ respectively, and get
	\begin{align}
		& D_{\omega}^{(k)}\breve{\mathcal{R}}_{2}^{(k)}(t_{2},x) - D_{\omega}^{(k)}\breve{\mathcal{R}}_{2}^{(k)}(t_{1},x) \notag \\
		=&D_{\omega}^{(k)}\breve{\mathcal{R}}_{2}^{(k)}(\tau_{2}^{(k)}(t_{2},x), \xi_{2}^{(k)}(t_{2},x)) -D_{\omega}^{(k)}\breve{\mathcal{R}}_{2}^{(k)}(\tau_{2}^{(k)}(t_{1},x), \xi_{2}^{(k)}(t_{1},x)) \notag \\
		&-\int_{\xi_{2}^{(k)}(t_{2},x)}^{x} D_{2}^{(k)}\lambda_{2}^{(k-1)}\partial_{x}\breve{\mathcal{R}}_{2}^{(k)} \Big|_{(T_{2}^{(k)}(y;t_{1},x),y)}^{(T_{2}^{(k)}(y;t_{2},x),y)}dy - \int_{\xi_{2}^{(k)}(t_{1},x)}^{\xi_{2}^{(k)}(t_{2},x)} D_{2}^{(k)}\lambda_{2}^{(k-1)}\partial_{x}\breve{\mathcal{R}}_{2}^{(k)} (T_{2}^{(k)}(y;t_{1},x),y)dy \notag \\
		& +\int_{\xi_{2}^{(k)}(t_{2},x)}^{x} \Lambda^{(k)} \Big|_{(T_{2}^{(k)}(y;t_{1},x),y)}^{(T_{2}^{(k)}(y;t_{2},x),y)}dy + \int_{\xi_{2}^{(k)}(t_{1},x)}^{\xi_{2}^{(k)}(t_{2},x)} \Lambda^{(k)} (T_{2}^{(k)}(y;t_{1},x),y)dy, \label{modulusInt}
	\end{align}
	where we use the notation for simplicity
	\begin{equation*}
		f\big|_{(t_{1},x_{1})}^{(t_{2},x_{2})} = f(t_{2},x_{2}) - f(t_{1},x_{1}),
	\end{equation*}
	for some function $f=f(t,x)$, and
	\begin{equation*}
		\Lambda^{(k)} :=  \kappa\alpha_{2}D_{\omega}^{(k)}\overline{E}^{(k-1)}+ \alpha_{2}D_{2}^{(k)}g_{2}.
	\end{equation*}
	By the definition \eqref{dOmega}, one can see that
	\begin{multline}\label{triangleIneq}
		\left|\partial_{t}\breve{\mathcal{R}}_{2}^{(k)}(t_{2},x) - \partial_{t}\breve{\mathcal{R}}_{2}^{(k)}(t_{1},x) \right| \le \left|D_{\omega}^{(k)}\breve{\mathcal{R}}_{2}^{(k)}(t_{2},x) - D_{\omega}^{(k)}\breve{\mathcal{R}}_{2}^{(k)}(t_{1},x)\right| \\
		+ \left|\omega_{2}^{(k)}(t_{2},x)\right|\left|\partial_{x}\breve{\mathcal{R}}_{2}^{(k)}(t_{2},x) - \partial_{x}\breve{\mathcal{R}}_{2}^{(k)}(t_{1},x) \right| + \left|\omega_{2}^{(k)}(t_{2},x) - \omega_{2}^{(k)}(t_{1},x)\right|\left|\partial_{x}\breve{\mathcal{R}}_{2}^{(k)}(t_{1},x)\right|.
	\end{multline}
	Here, combining \eqref{modulusDissipation}, \eqref{omegaC1}, \eqref{T1T2esitm} and \eqref{X1X2estim}--\eqref{triangleIneq}, as well as $C^{1}$ estimates \eqref{C1estima}, one can get that for large enough $C_{M,1}>0$,
	\begin{equation}\label{modulusEstiX}
		\varpi(\eta|\frac{\partial \breve{\mathcal{R}}_{2}^{(k)}}{\partial t}(\cdot,x)) \le \theta_{0}C_{M,1}\Phi(\eta) + C\epsilon\varpi(\eta|\frac{\partial \breve{\mathcal{R}}_{2}^{(k)}}{\partial x}(\cdot,x)),
	\end{equation}
	where $\theta_{0}$ is some constant satisfying $\alpha_{2}\theta_{d} < \theta_{0} <1$.
	\par Finally, using the equation \eqref{liearSys2} and noting the known estimates \eqref{C1estima} and \eqref{c66}, one can derive
	\begin{equation}\label{modulusXfromT}
		\varpi(\eta|\frac{\partial \breve{\mathcal{R}}_{2}^{(k)}}{\partial x}(\cdot,x))
		\leq \mu_{\max}\varpi(\eta|\frac{\partial \breve{\mathcal{R}}_{2}^{(k)}}{\partial t}(\cdot,x))+(C\epsilon+C\kappa)\Phi(\eta),
	\end{equation}
	which, together with \eqref{modulusEstiX}, completes the proof of \eqref{c63} for sufficiently large $C_{M,1}$, $C_{M,2}$.
	
	\item[Step 4. Uniform continuity modulus estimates] Finally, we turn to the proof of \eqref{c64}. We consider two points $(t_{1},x_{1})$, $(t_{2},x_{2})$ with $|t_{1}-t_{2}|\le \eta$ and $|x_{1}-x_{2}|\le \eta$ located on the same characteristic $T_{2}^{(k)}(y;t_{1},x_{1})$, namely,
	\begin{equation*}
		t_{2} = T_{2}^{(k)}(x_{2};t_{1},x_{1}).
	\end{equation*}
	For this case, one can integrate the equation \eqref{deriTrans} along the characteristic to get
	\begin{align*}
		D_{\omega}^{(k)}\breve{\mathcal{R}}_{2}^{(k)}(t_{2},x_{2}) - D_{\omega}^{(k)}\breve{\mathcal{R}}_{2}^{(k)}(t_{1},x_{1}) = \int_{x_{1}}^{x_{2}}\left(-D_{2}^{(k)}\lambda_{2}^{(k-1)}\partial_{x}\breve{\mathcal{R}}_{2}^{(k)}+\Lambda^{(k)}\right)(T_{2}^{(k)}(y;t_{1}, x_{1}),y)dy.
	\end{align*}
	And thus, with \eqref{pFpx}, \eqref{C1estima} and \eqref{c66}, one has
	\begin{equation}\label{sameCharacter}
		\left|D_{\omega}^{(k)}\breve{\mathcal{R}}_{2}^{(k)}(t_{2},x_{2}) - D_{\omega}^{(k)}\breve{\mathcal{R}}_{2}^{(k)}(t_{1},x_{1})\right| \le CC_{M}(\epsilon+\kappa)\Phi(\eta).
	\end{equation}
	\par For general two points $(t_{1},x_{1})$, $(t_{2},x_{2})$ with $|t_{1}-t_{2}|\le \eta$ and $|x_{1}-x_{2}|\le \eta$, one can find a point $(t_{3},x_{2})$ located on the characteristic $T_{2}^{(k)}(y;t_{1},x_{1})$. By \eqref{characteriticT}, one has
	\begin{equation*}
		\left|t_{3}-t_{1}\right| \le \mu_{\max}\left|x_{2}-x_{1}\right| \le \mu_{\max}\eta,
	\end{equation*}
	and thus
	\begin{equation}\label{muandone}
		\left|t_{3}-t_{2}\right| \le  \left(\mu_{\max}+1\right)\eta.
	\end{equation}
	Therefore, by \eqref{c63}, \eqref{sameCharacter} and \eqref{muandone}, one has
	\begin{align}
		&\left|D_{\omega}^{(k)}\breve{\mathcal{R}}_{2}^{(k)}(t_{2},x_{2}) - D_{\omega}^{(k)}\breve{\mathcal{R}}_{2}^{(k)}(t_{1},x_{1})\right| \notag\\
		\le & \left|D_{\omega}^{(k)}\breve{\mathcal{R}}_{2}^{(k)}(t_{2},x_{2}) - D_{\omega}^{(k)}\breve{\mathcal{R}}_{2}^{(k)}(t_{3},x_{2})\right| + \left|D_{\omega}^{(k)}\breve{\mathcal{R}}_{2}^{(k)}(t_{3},x_{2}) - D_{\omega}^{(k)}\breve{\mathcal{R}}_{2}^{(k)}(t_{1},x_{1})\right| \notag \\
		\le &\left(\mu_{\max}+2\right)C_{M,1}\Phi(\eta) + CC_{M}(\epsilon+\kappa)\Phi(\eta).
	\end{align}
	Using the similar argument as \eqref{triangleIneq} and \eqref{modulusXfromT}, one can conclude
	\begin{equation}
		\varpi(\eta|\frac{\partial \breve{\mathcal{R}}_{2}^{(k)}}{\partial t}(\cdot,\cdot)) \le \left(\mu_{\max}+3\right)C_{M,1}\Phi(\eta),\quad \varpi(\eta|\frac{\partial \breve{\mathcal{R}}_{2}^{(k)}}{\partial x}(\cdot,\cdot)) \le \mu_{\max}\left(\mu_{\max}+3\right)C_{M,1}\Phi(\eta).
	\end{equation}
	Choosing $C_{M} = \left(\mu_{\max}+2\right)^{2}C_{M,1}$, the proof of \eqref{c64} is completed. \qedhere
\end{description}
\end{proof}

\section{Dynamical Stability of the Time--periodic Transonic Shock Solution}\label{s4}
In this section, we prove our main \Cref{t2}. To begin with, it follows from the argument in \cite{Yuw} that the initial--boundary value problem~\eqref{a1}--\eqref{a5} with the assumptions \eqref{a8}--\eqref{AA11} and initial data \eqref{a19}--\eqref{a20} admits a semi--global unique transonic shock solution $(\rho, u, E)(t,x)$ with shock position $\sigma(t)$ satisfying
\begin{align}
\sum_{m=0}^{1}|\partial_{t}^{m}(\sigma(t)-\tilde{x})|+\|(\overline{\rho}_{-},\overline{u}_{-},\overline{E}_{-})(t,\cdot)
\|_{C^{1}([0,\sigma(t)))}
&+\|\breve{\bm{\mathcal{R}}}(t,\cdot)\|_{C^{1}((\sigma(t),L])}
<C\epsilon,\label{d1}
\end{align}
where $$\overline{\rho}_{-}(t,x)=\rho_{-}(t,x)-\tilde{\rho}_{-}(x),~\overline{u}_{-}(t,x)=u_{-}(t,x)-\tilde{u}_{-}(x),
~\overline{E}_{-}(t,x)=E_{-}(t,x)-\tilde{E}_{-}(x),$$
and $\breve{\bm{\mathcal{R}}} = \left(\breve{\mathcal{R}}_{1}, \breve{\mathcal{R}}_{2}\right)^{\top}$ is the solution to \eqref{sys57}--\eqref{system3} with boundary conditions \eqref{shockBoundary2}--\eqref{outletBoundary1}.
\par Moreover, we denote the upper bound for the time required for a wave to reach the boundary by
\begin{equation}\label{timeMax}
	\mathcal{T}_{0}=\mu_{\max}L,
\end{equation}
where $\mu_{\max}$ is defined in \eqref{mumax}. Then in the supersonic region, the exponential decay follows from the argument in~\cite{Yuan}, namely, for $t\in[N\mathcal{T}_{0},(N+1)\mathcal{T}_{0}],~N\in\mathbb{Z}_{+}$,
\begin{align}
	\|(\overline{\rho}_{-},\overline{u}_{-})(t,\cdot)-(\overline{\rho}_{-}^{(\mathcal{T})},\overline{u}_{-}^{(\mathcal{T})})(t,\cdot)\|
	_{C^{0}([0,\min\{\sigma(t),\sigma^{(\mathcal{T})}(t)\}))}\leq C_{L,0}\epsilon\zeta^{N}.\label{d0}
\end{align}
And together with \eqref{barEle}, it follows
\begin{align}
\|(\overline{\rho}_{-},\overline{u}_{-},\overline{E}_{-})(t,\cdot)-(\overline{\rho}_{-}^{(\mathcal{T})},\overline{u}_{-}^{(\mathcal{T})},
\overline{E}_{-}^{(\mathcal{T})})(t,\cdot)\|
_{C^{0}([0,\min\{\sigma(t),\sigma^{(\mathcal{T})}(t)\}))}\leq C_{L,1}\epsilon\zeta^{N}.\label{d2}
\end{align}
Here $C_{L,0}, C_{L,1}$ are positive constants.
\par For the problem in the subsonic region, we denote by 
\begin{equation*}
	\breve{\bm{\mathcal{R}}}^{(\mathcal{T})}(t,x) = \left(\breve{\mathcal{R}}_{1}^{(\mathcal{T})}, \breve{\mathcal{R}}_{2}^{(\mathcal{T})}\right)(t,x), \quad
	\breve{\bm{\mathcal{R}}}(t,x) = \left(\breve{\mathcal{R}}_{1}, \breve{\mathcal{R}}_{2}\right)(t,x),
\end{equation*}
the solutions to system \eqref{a1}--\eqref{a5} with the initial data \eqref{a14} and \eqref{a19} respectively in the corresponding regions
\begin{equation*}
	\Omega_{+}^{\mathcal{T}} = \left\{(t,x):t\in \mathbb{R}_{+}, \sigma^{(\mathcal{T})}(t)\le x \le L \right\}, \quad \Omega_{+} := \left\{(t,x):t\in \mathbb{R}_{+}, \sigma(t)\le x \le L \right\}.
\end{equation*}
In order to prove \eqref{a21}, with \eqref{d2} in hand, it remains to prove inductively that,
for some $t_{0}>0$ and $N \in \mathbb{Z}_{+}$, 
\begin{align}
	&\|\breve{\bm{\mathcal{R}}}(t,\cdot)-\breve{\bm{\mathcal{R}}}^{(\mathcal{T})}(t,\cdot)\| \le C_{S,1}\epsilon\zeta^{N+1}, \quad \forall t \in [t_{0}+\mathcal{T}_{0},t_{0}+2\mathcal{T}_{0}], \\
	&|\sigma(t)
	-\sigma^{(\mathcal{T})}(t)| \le C_{S,2}\epsilon\zeta^{N+1}, \quad |\sigma^{\prime}(t)
	-{\sigma^{(\mathcal{T})}}^{\prime}(t)| \le C_{S,2}\epsilon\zeta^{N+1}, \quad \forall t \in [t_{0}+\mathcal{T}_{0},t_{0}+2\mathcal{T}_{0}],
\end{align}
under the assumptions
\begin{align}
	&\|\breve{\bm{\mathcal{R}}}(t,\cdot)-\breve{\bm{\mathcal{R}}}^{(\mathcal{T})}(t,\cdot)\| \le C_{S,1}\epsilon\zeta^{N}, \quad \forall t \in [t_{0},t_{0}+\mathcal{T}_{0}], \\
	&|\sigma(t)
	-\sigma^{(\mathcal{T})}(t)| \le C_{S,2}\epsilon\zeta^{N}, \quad |\sigma^{\prime}(t)
	-{\sigma^{(\mathcal{T})}}^{\prime}(t)| \le C_{S,2}\epsilon\zeta^{N}, \quad \forall t \in [t_{0},t_{0}+\mathcal{T}_{0}].
\end{align}
To this end, we use a bootstrap argument as follows. Denote
\begin{align}
	&\Gamma_{1}(t)=\mathop{\max}\limits_{i=1,2}\mathop{\sup}\limits_{x\in[\max\{\sigma(t),\sigma^{(\mathcal{T})}(t)\},L]}
	|\breve{\mathcal{R}}_{i}(t,x)
	-\breve{\mathcal{R}}^{(\mathcal{T})}_{i}(t,x)|, \notag \\
	&\Gamma_{2}(t)=
	|\sigma(t)
	-\sigma^{(\mathcal{T})}(t)|, \qquad \Gamma_{3}(t)=
	|\sigma^{\prime}(t)
	-{\sigma^{(\mathcal{T})}}^{\prime}(t)|. \notag
\end{align}
By the regularity of $\breve{\bm{\mathcal{R}}}$ and $\sigma$, one can see that $\Gamma_{1}(t)$, $\Gamma_{2}(t)$ and $\Gamma_{3}(t)$ are continuous and satisfy
\begin{equation*}
	\Gamma_{1}(t_{0}+\mathcal{T}_{0}) \le C_{S,1}\epsilon\zeta^{N}, \quad \Gamma_{2}(t_{0}+\mathcal{T}_{0}) \le C_{S,2}\epsilon\zeta^{N}, \quad \Gamma_{3}(t_{0}+\mathcal{T}_{0}) \le C_{S,2}\epsilon\zeta^{N}.
\end{equation*}
Thus, it is sufficient to show  that for each given $t_* \in [t_{0}+\mathcal{T}_0, t_{0}+2\mathcal{T}_0]$, 
\begin{equation}\label{bootstrapGoal}
	\Gamma_{1}(t) \le C_{S,1}\epsilon\zeta^{N+1}, \quad \Gamma_{2}(t) \le C_{S,2}\epsilon\zeta^{N+1}, \quad \Gamma_{3}(t) \le C_{S,2}\epsilon\zeta^{N+1}, \quad \forall t\in [t_{0}+\mathcal{T}_{0}, t_{*}],
\end{equation}
under the assumptions
\begin{equation}\label{bootstrapAssum}
	\Gamma_{1}(t) \le \alpha C_{S,1}\epsilon\zeta^{N}, \quad \Gamma_{2}(t) \le  \alpha C_{S,2}\epsilon\zeta^{N}, \quad \Gamma_{3}(t) \le \alpha C_{S,2}\epsilon\zeta^{N}, \quad \forall t\in [t_{0}, t_{*}],
\end{equation}
for some constant $\alpha>1$ to be determined later.
\par Now, we use the equations \eqref{system3} for $\breve{\mathcal{R}}_{2}$ and $\breve{\mathcal{R}}_{2}^{(\mathcal{T})}$ respectively to get
\begin{align}
	&\left(\partial_{t}+\lambda_{2}\partial_{x}\right)\left(\breve{\mathcal{R}}_{2}- \breve{\mathcal{R}}_{2}^{(\mathcal{T})}\right) \notag \\
	=& - \left(\lambda_{2}-\lambda_{2}^{(\mathcal{T})}\right)\partial_{x}\breve{\mathcal{R}}_{2}^{(\mathcal{T})} +\alpha_{2}\left(g_{2}(x,\breve{\bm{\mathcal{R}}})-g_{2}(x,\breve{\bm{\mathcal{R}}}^{(\mathcal{T})})\right) \notag \\
	&+ \kappa\alpha_{2}\int_{\sigma^{(\mathcal{T})}(t)}^{\sigma(t)}\left(\rho_{l}(t, y) - \rho(y, \breve{\bm{\mathcal{R}}}^{(\mathcal{T})}(t,y))\right)dy + \kappa\alpha_{2}\int_{\sigma(t)}^{x}\left(\rho(y, \breve{\bm{\mathcal{R}}}(t,y)) - \rho(y, \breve{\bm{\mathcal{R}}}^{(\mathcal{T})}(t,y))\right)dy. \label{d10}
\end{align}
\par Next, one can define the characteristic as
\begin{align}\label{characteritic0}
	\left\{
	\begin{aligned}
		&\frac{dX_{2}}{d\tau}(\tau;t,x)=\lambda_{2}\left(X_{2}(\tau;t,x),\breve{\bm{\mathcal{R}}}\left(\tau,X_{2}(\tau;t,x)\right)\right),\\
		&X_{2}(t;t,x)=x,
	\end{aligned}
	\right.
\end{align}
for any fixed $(t,x) \in [t_{0}+\mathcal{T}_{0},t_{*}]\times[0,L]\cap\Omega_{+}^{\mathcal{T}}\cap\Omega_{+}$. Since the solution is $C^{1}$ smooth, the characteristic is well--defined. In the same way, one can define the characteristic $X_{2}^{(\mathcal{T})}$ for the time--periodic solution. In addition, one can define the backward exit--time for $(t,x)$ as respectively $\tau_{2}(t,x)$/$\tau_{2}^{(\mathcal{T})}(t,x)$ by the intersection of the characteristic $X_{2}$/$X_{2}^{(\mathcal{T})}$ and the shock $\sigma$/$\sigma^{(\mathcal{T})}$, i.e.,
\begin{equation}\label{exitTime0}
	X_{2}(\tau_{2};t,x)=\sigma(\tau_{2}), \quad X_{2}^{(\mathcal{T})}(\tau_{2}^{(\mathcal{T})};t,x)=\sigma^{(\mathcal{T})}(\tau_{2}^{(\mathcal{T})}).
\end{equation}
By \eqref{timeMax}, one has $\{\tau_{2}, \tau_{2}^{(\mathcal{T})}\} \subset [t_{0}, t_{*}]$.
\par Furthermore, one can integrate the ODE \eqref{shockposition} from $t_{0}$ to $t$ and use Hadamard's formula as well as \eqref{pFpx} to get
\begin{multline}\label{smallZeta}
	\left(1+2\overline{C}\kappa\mathcal{T}_{0}\right)\Gamma_{2}(t) \le \alpha C_{S,2}\epsilon\zeta^{N} 
	+2\mathcal{T}_{0}(1+C\epsilon+C\kappa)\left(\sum_{i=1}^{2}\dfrac{1}{\alpha_{i}}|\frac{\partial \mathcal{F}}{\partial \overline{\mathcal{R}}_{i}}(\tilde{x},\bm{0})|\Gamma_{1}(t)\right. \\
	+\left. \left(|\frac{\partial \mathcal{F}}{\partial \overline{\rho}_{l}}(\tilde{x},\bm{0})|
	+|\frac{\partial \mathcal{F}}{\partial \overline{u}_{l}}(\tilde{x},\bm{0})|\right)\left(C_{l}\epsilon\Gamma_{2}(t) + C_{L}\epsilon\zeta^{N}\right)\right).
\end{multline}
\par Then, for fixed $(t,x)$, one can integrate \eqref{d10} along the characteristic $X_{2}(\tau;t,x)$. Without loss of generality, assume that $X_{2}(\tau;t,x)$ intersects $\sigma$ (rather than $\sigma^{(\mathcal{T})}$) at the point $(\tau_{2}, \sigma(\tau_{2}))$. Applying the very similar argument as \eqref{c55}--\eqref{cauchy}, and further noting \eqref{d2}, one can conclude
\begin{equation}
	\Gamma_{1}(t) \le (1+C\epsilon +C\kappa )\alpha_{2}\theta_{d}\alpha C_{S,1}\epsilon\zeta^{N} +CC_{L}\epsilon\zeta^{N} + \alpha CC_{S,1}\epsilon^{2}\zeta^{N}, \quad \forall t\in [t_{0}+\mathcal{T}_{0},t_{*}].
\end{equation}
We choose $\zeta$ such that
\begin{equation*}
	(1+C\epsilon +C\kappa )\alpha_{2}\theta_{d}\alpha < \zeta < 1,
\end{equation*}
such that for large enough $C_{S,1}$, 
\begin{equation}\label{Gamma1}
	\Gamma_{1}(t) \le C_{S,1}\epsilon\zeta^{N+1}, \quad \forall t\in [t_{0}+\mathcal{T}_{0},t_{*}].
\end{equation}
The analysis for $\breve{\mathcal{R}}_{1}$ is similar. 
\par Moreover, we choose $\zeta$ and $\alpha$ such that
\begin{equation*}
	\alpha\left(1+2\overline{C}\kappa\mathcal{T}_{0}\right)^{-1} < \zeta <1,
\end{equation*}
and then the estimate for $\Gamma_{2}(t)$ in \eqref{bootstrapGoal} follows from \eqref{smallZeta}. Finally, the estimate for $\Gamma_{3}(t)$ in \eqref{bootstrapGoal} can be derived directly from \eqref{shockposition} and \eqref{Gamma1} by using Hadamard's formula. The proof of \Cref{t2} is now complete.

\section*{Acknowledgments}
Qu and Wang are supported in part by NSFC Grants No. 12431007 and 62588101.

\section*{Declarations}
\subsection*{Data availability statement}
No data was used for the research described in this article.

\subsection*{Conflict of interest}
The authors have no competing interests to declare that are relevant to the content of this article.

\small
\begin{spacing}{1}
	\bibliographystyle{abbrv}
	\bibliography{myShockBibs2}
\end{spacing}
\end{document}